\documentclass{article}
\usepackage[utf8]{inputenc}
\usepackage{amsmath,amssymb,amsthm,amsfonts,latexsym,graphicx,subfigure,mathtools,enumitem}

\usepackage{url}
\usepackage{xcolor}
\usepackage{svg}

\newcommand{\R}{\mathbb{R}}
\newcommand{\RQ}{\textrm{RQ}}

\DeclareMathOperator{\id}{Id}

\DeclareMathOperator{\supp}{supp}

\DeclareMathOperator{\diag}{diag}

\theoremstyle{plain}
\newtheorem{theorem}{Theorem}[section]
\newtheorem{lemma}[theorem]{Lemma}
\newtheorem{proposition}[theorem]{Proposition}
\newtheorem{corollary}[theorem]{Corollary}

\newtheorem{question}{Question}

\theoremstyle{definition}
\newtheorem{definition}[theorem]{Definition}
\newtheorem{example}[theorem]{Example}
\theoremstyle{remark}
\newtheorem{remark}[theorem]{Remark}

\usepackage[affil-it]{authblk}
\usepackage{booktabs}   % for \toprule et al
\usepackage[numbers,sort&compress]{natbib}

\usepackage{hyperref}

\usepackage[margin=3.5cm]{geometry}

\title{At the end of the spectrum: \\ Chromatic bounds for the largest eigenvalue of the normalized Laplacian}
\author[1]{Lies Beers\thanks{e.g.m.beers@vu.nl}}
\author[1,2]{Raffaella Mulas}
\affil[1]{Vrije Universiteit Amsterdam, Amsterdam, The Netherlands}
\affil[2]{Max Planck Institute for Mathematics in the Sciences, Leipzig, Germany}

\date{}

\allowdisplaybreaks

\begin{document}

\maketitle

\begin{abstract} For a graph with largest normalized Laplacian eigenvalue $\lambda_N$ and (vertex) coloring number $\chi$, it is known that $\lambda_N\geq \chi/(\chi-1)$. Here we prove properties of graphs for which this bound is sharp, and we study the multiplicity of $\chi/(\chi-1)$. We then describe a family of graphs with largest eigenvalue $\chi/(\chi-1)$. We also study the spectrum of the $1$-sum of two graphs (also known as graph joining or coalescing), with a focus on the maximal eigenvalue. Finally, we give upper bounds on $\lambda_N$ in terms of $\chi$.

\vspace{0.2cm}

\end{abstract}

\tableofcontents

\section{Introduction} 

\subsection{At the end of the color spectrum}
As anyone who is familiar with more than one language knows, there are some things that you can express precisely in one language that you cannot in another, and vice versa. In many aspects, the evolution of different languages is somewhat arbitrary. There are, on the other hand, also universalities in the development of different languages, a notable one being basic color terms. In 1969, Berlin and Kay \cite{berlin1969basic} identified different stages in the development of basic color terms for distinct languages. In Stage I, a distinction is developed between darker (\emph{black}) and lighter shades (\emph{white}). In Stage II, a word for the color \emph{red} and adjacent shades is developed, and in Stage III, either a word for \emph{yellow} or \emph{green} shades emerges; in Stage IV, there are words for both yellow and green. This evolution continues: in Stage V, the word for \emph{blue} is coined, in Stage VI, the word for \emph{brown} appears, and in the final stages, four colors are added, not necessarily in a fixed order,
until there are exactly eleven terms for the basic colors: black, white, red, green, yellow, blue, brown, purple, pink, orange and grey.

Berlin and Kay studied which of the eleven basic color terms were in the vocabulary of different languages from different families, and encountered only $22$ combinations of these eleven colors, out of $2048$ possible combinations.
It thus appears that the evolution of basic color terms is a constant in the development of a language: following more or less the same pattern, the same eleven basic color terms appear in the vocabulary of any language, to describe a spectrum of colors.

A field in which both colors and spectra are studied is graph theory: in spectral graph theory, spectra of different matrices associated with a graph are studied, and in chromatic graph theory \cite{ChartrandZhang2009chromatic,beineke2015topics}, different notions of colorings of a graph are studied. In this paper, we combine these branches of graph theory, by studying bounds relating the largest eigenvalue of the normalized Laplacian of a graph to its vertex coloring number.

\subsection{At the end of the Laplacian spectrum}

The context in which we shall work is the following. Let $G=(V,E)$ be a finite \emph{simple graph}, i.e., an undirected, unweighted graph without multi-edges and without loops, with $N$ vertices $v_1,\ldots,v_N$. Two distinct vertices $v$ and $w$ are called \emph{adjacent}, denoted $v\sim w$ or $w\sim v$, if $\{v,w\}\in E$. The \emph{degree} $\deg_G v$ or $\deg v$ of a vertex $v$ is the number of vertices that it is adjacent to. If the degree of every vertex of $G$ is equal to some fixed positive integer $d$, then we say that $G$ is \emph{regular} or \emph{$d$-regular}. We assume that no vertex has degree $0$.

A \emph{(vertex) $k$-coloring} is a function $c:V\to \{1,\ldots,k\}$, and it is \emph{proper} if $v\sim w$ implies that $c(v)\neq c(w)$. The \emph{(vertex) coloring number} or \emph{chromatic number} $\chi=\chi(G)$ is the minimum $k$ such that there exists a proper $k$-coloring of the vertices. If $\chi(G)\leq 2$, then we say that $G$ is \emph{bipartite}. These and other elementary definitions in graph theory can be found, for instance, in \cite{GodsilRoyle2001}.

When determining the coloring number $\chi$ of a given graph $G$, one can find an upper bound $\chi\leq k$ by giving a proper $k$-coloring of the vertices of $G$. Finding lower bounds directly, by showing that a graph cannot admit a proper $k$-coloring for some $k$, requires more work in general. Therefore, it is useful to find lower bounds in another way, for example, by considering the spectrum of a matrix associated with $G$.

Given a graph $G$, four notable matrices whose spectra are studied in spectral graph theory, are the adjacency matrix, the Kirchoff Laplacian, the signless Laplacian, and the normalized Laplacian. 
The \emph{adjacency matrix} of $G$ is the $N\times N$ matrix $A \coloneqq A(G)$ with entries 
\begin{equation*}
    A_{ij} \coloneqq\begin{cases}1 &\text{ if }v_i\sim v_j\\
    0 &\text{ otherwise.}\end{cases}
\end{equation*}  
The \emph{Kirchoff Laplacian} is the matrix $K\coloneqq K(G)\coloneqq D-A$, and the \emph{signless Laplacian} is the matrix $Q\coloneqq Q(G)\coloneqq D+A$, where $D\coloneqq \diag\bigl(\deg v_1,\ldots,\deg v_N\bigr)$. The spectrum of the adjacency matrix, especially for regular graphs, and
of the Kirchoff Laplacian have been studied, for instance, in \cite{GodsilRoyle2001,CvetkovicDragos2010spectra,brouwer2011spectra}. Literature on the signless Laplacian can be found in \cite{cvetkovic2010spectral}.

In 1992, Chung \cite{chung} introduced the matrix
$$\mathcal{L}:=\mathcal{L}(G):=\id-D^{-1/2}AD^{-1/2},$$ where $\id$ denotes the $N\times N$ identity matrix. 
The matrix $\mathcal{L}$ is \emph{similar} (in matrices terms) to the \emph{normalized Laplacian} of $G$, which is defined as
$$L:=L(G):=\id-D^{-1}A.$$ In fact, $L=D^{-1/2}\mathcal{L} D^{1/2}$.
Therefore, given a graph $G$, the spectra of $L$ and $\mathcal L$ coincide. Here we shall focus on $L$. Its entries are 
\begin{equation*}
    L_{ij}=\begin{cases}1 & \text{ if }i=j\\ -\frac{1}{\deg v_i} & \text{ if }v_i\sim  v_j\\ 0 & \text{ otherwise.}\end{cases}
\end{equation*}In particular, for $v_i\sim v_j$, $-L_{ij}$ is the probability of going from $v_i$ to $v_j$ with a classical random walk on $V$.

The spectra of the four matrices $A$, $K$, $Q$, or $L$ can be used to find different information about $G$. For example, from the spectrum of the adjacency matrix, one can derive the number of edges of $G$, which is not possible from the normalized Laplacian spectrum, whereas the multiplicity of the eigenvalue $0$ in the normalized Laplacian spectrum equals the number of connected components, which is information that the adjacency spectrum cannot give you.

Sometimes one language provides you with the words to say something that you cannot say in another. Similarly, some graphs have the same spectrum with respect to one matrix\,---\,we say that they are \emph{cospectral} with respect to that matrix\,---\,but they have different spectra with respect to another. For example, all complete bipartite graphs with the same number of vertices have the same normalized Laplacian spectrum, but not the same adjacency spectrum.

In the same way that the evolution of basic color terms is a constant factor in the development of a language, we have that $d$-regular graphs are in some way a constant factor in spectral analysis of the aforementioned graphs: if a graph $G$ is $d$-regular, then the spectrum of one among $A$, $K$, $Q$ and $L$, determines the spectrum of the other three. In fact, in this case, we have that $D = d\cdot \id$, implying that $K=d\cdot \id-A$ and $\mathcal{L}=L=\frac{1}{d}\cdot K$. Hence, for $d$-regular graphs,
\begin{align*}
    \lambda \text{ is an eigenvalue for }K 
    &\iff d-\lambda \text{ is an eigenvalue for }A\\
    &\iff \frac{\lambda}{d} \text{ is an eigenvalue for }\mathcal{L}=L.
\end{align*}
This also implies that two $d$-regular graphs $G_1$ and $G_2$ are cospectral with respect to one matrix if and only if they are cospectral with respect to another.

Problems in spectral graph theory include
finding cospectral graphs, as well as finding graphs that are determined by their spectrum, with respect to some of the four matrices that we defined. Other questions concern themselves with relating the spectrum of one of the matrices to other graph properties. An example of a well-known result 
is the Hoffman bound \cite{Hoffman1970}, which gives a lower bound on the vertex coloring number using the smallest and largest adjacency eigenvalues. This bound has been generalized to include more eigenvalues \cite{generalizedHoffman}.

We are, in this paper, interested in the normalized Laplacian spectrum. We let $$\lambda_N\geq\lambda_{N-1}\geq\cdots\geq \lambda_1=0$$ denote the eigenvalues of $L$, and we also introduce the notation $\lambda_{\max}:=\lambda_N$.
We have that $0\leq\lambda \leq 2$ for every eigenvalue $\lambda$ of $L$.
The multiplicity of the eigenvalue $0$ equals the number of connected components, and the multiplicity of the eigenvalue $2$ equals the number of bipartite components. 
More background on the normalized Laplacian spectrum can be found in \cite{chung,ButlerChung2006,Cavers2010,Butler2016}.

Since the normalized Laplacian spectrum of a graph equals the union of the spectra of its connected components,
we assume for the rest of the paper that $G$ is connected. We also assume that $N\geq2$.
The normalized Laplacian eigenvalues that are studied the most, are the second smallest and the largest. It is known, for example, that the largest eigenvalue equals $2$ if and only if $G$ is bipartite, while it is equal to $N/(N-1)$ if and only if $G$ is the complete graph. For all other graphs, we have that $\lambda_N\geq (N+1)/(N-1)$ \cite{SunDas2016,JMM}.
Some other problems involving the normalized Laplacian 
regard multiplicities, for example: Which graphs have two normalized Laplacian eigenvalues, and which graphs have three normalized Laplacian eigenvalues \cite{VanDamOmidi2011}? The answer to the first question is: only complete graphs, while the answer to the second question is not known.

Other questions include: Which graphs are determined by their spectrum?
Which graphs have an eigenvalue with multiplicity $N-2$ \cite{VanDamOmidi2011}, and which graphs have an eigenvalue with multiplicity $N-3$ \cite{SunDas2021,TianWang2021}? How do the eigenvalues change when deleting an edge of $G$ \cite{butler-interlacing}? How does the spectrum change under other graph operations \cite{ChenLiao2017}?

As we saw before, we have, given $N$, that bipartite graphs have the largest possible largest eigenvalue, whereas the complete graph has the smallest possible largest eigenvalue. In some other sense, bipartite graphs and complete graphs are also on opposite ends of a spectrum:
the former has the smallest possible coloring number, and the latter has the largest possible coloring number given $N$. In both cases, we have that $\lambda_N = \chi/(\chi-1)$.
Elphick and Wocjan (2015) \cite{ElphickWocjan2015} (Equation 20) proved that, in general, we have the inequality
$$\lambda_N\geq \frac{\chi}{\chi-1},$$ which coincides with the Hoffman bound for regular graphs.

In this paper, we study graphs for which this inequality is sharp. These graphs are special, as they relate to some of the aforementioned problems, regarding the smallest possible value of $\lambda_N$ in terms of $N$, and graphs with a largest eigenvalue of multiplicity $N-2$ and $N-3$.

\section{Background}\label{section:background}

\subsection{Basic definitions, notations and properties}
In this section we shall introduce some more definitions, notations and properties that we shall refer to throughout the paper. As in the Introduction, we fix a simple graph $G=(V,E)$ on $N$ vertices, we assume that $G$ is connected, and we let $v_1,\ldots,v_N$ denote its vertices.

We start by listing several properties of the normalized Laplacian of $G$ and its spectrum.
\begin{remark}
Let $C(V)$ denote the vector space of functions $f:V\rightarrow\mathbb{R}$ and, given $f,g\in C(V)$, let
	\begin{equation*}
	\langle f,g\rangle:=\sum_{v\in V}\deg v\cdot f(v)\cdot g(v).
	\end{equation*}
We can see the normalized Laplacian $L$ as an operator $C(V)\rightarrow C(V)$ such that
\begin{equation}\label{eq:Lf}
    Lf(v)=f(v)-\frac{1}{\deg v}\sum_{w\sim v}f(w).
\end{equation}Also, it is easy to check that $L$ is self-adjoint with respect to the inner product $	\langle \cdot,\cdot \rangle$, i.e.,
	\begin{equation*}
	\langle Lf,g\rangle=\langle f,Lg\rangle \quad \forall f,g\in C(V).
	\end{equation*}
\end{remark}

\begin{remark}
    By \eqref{eq:Lf}, $(\lambda,f)$ is an eigenpair for $L$ if and only if, for all $v\in V$,
    \begin{equation*}
    \lambda f(v)=f(v)-\frac{1}{\deg v}\sum_{w\sim v}f(w),
\end{equation*}which can be equivalently rewritten as
\begin{equation}\label{eq:eigenpair}
    (1-\lambda) f(v)=\frac{1}{\deg v}\sum_{w\sim v}f(w).
\end{equation}
\end{remark}

With the Courant-Fischer-Weyl min-max Principle below, we can characterize the eigenvalues of $L$.

\begin{theorem}[Courant-Fischer-Weyl min-max Principle]
\label{min-max theorem}

Let $H$ be an $N$-dimensional vector space with a positive definite scalar product $(.,.)$, and let $A:H\rightarrow H$ be a self-adjoint linear operator. Let $\mathcal{H}_{k}$  be the  family of all $k$-dimensional subspaces of $H$.
  Then the eigenvalues $\lambda_{1}\leq \ldots \leq \lambda_{N}$ of
$A$ can be obtained by 
\begin{equation}
\label{min-max}
\lambda_{k}=\min_{H_k\in \mathcal{H}_{k}}\max_{g(\neq 0)\in H_{k}}\frac{(Ag,g)}{(g,g)}=\max_{{H}_{N-k+1} \in \mathcal{H}_{N-k+1}}\min_{g(\neq 0)\in {H}_{N-k+1}}\frac{(Ag,g)}{(g,g)}.
\end{equation}
The vectors $g_k$ realizing such a min-max or max-min then are corresponding
eigenvectors, and the min-max spaces $H_{k}$ are spanned
by the eigenvectors for the eigenvalues $\lambda_1,\dots
,\lambda_k$, and analogously, the max-min spaces $H_{N-k+1}$ are spanned
by the eigenvectors for the eigenvalues $\lambda_k, \dots, \lambda_N$.\\

Thus, we also have

\begin{equation}
\label{min-max1}
\lambda_{k}=\min_{g(\neq 0)\in H, (g,g_j)=0\ \mathrm{ for }\ j=1,\dots ,k-1}\frac{(Ag,g)}{(g,g)}=\max_{g(\neq 0)\in H, (g,g_\ell)=0\ \mathrm{ for }\ \ell=k+1,\dots ,N}\frac{(Ag,g)}{(g,g)}.
\end{equation}

In particular, 
\begin{equation}
\label{min-max2}
\lambda_{1}=\min_{g(\neq 0)\in H}\frac{(Ag,g)}{(g,g)},\qquad \lambda_N=\max_{g(\neq 0)\in H}\frac{(Ag,g)}{(g,g)}.
\end{equation}
\end{theorem}
\begin{definition}\label{RQ}
$(Ag,g)/(g,g)$ is called the \emph{Rayleigh quotient} of $g$. 
\end{definition}

According to Theorem \ref{min-max theorem}, the eigenvalues of $L$ are given by min-max values of
\begin{equation*}
    \RQ(f):=\frac{\langle Lf,f\rangle}{\langle f,f\rangle}=\frac{\sum_{v\sim w}\biggl(f(v)-f(w)\biggr)^2}{\sum_{v\in V}\deg v\cdot f(v)^2}, \quad \text{for }f\in C(V).
\end{equation*}In particular, let $k\in \{1,\ldots,N\}$ and let $g_i$ be eigenfunctions for $\lambda_i$, for each $i\in \{1,\ldots,N\}\setminus \{k\}$, that are pairwise linearly independent. Then,
\begin{equation*}
    \lambda_k=\min_{\substack{f\in C(V)\setminus\{\mathbf{0}\}:\\ \langle f,g_1\rangle=\ldots=\langle f,g_{k-1}\rangle=0}} \RQ(f)=\max_{\substack{f\in C(V)\setminus\{\mathbf{0}\}:\\ \langle f,g_{k+1}\rangle=\ldots=\langle f,g_N\rangle=0}} \RQ(f),
\end{equation*}and the functions realizing such a min-max are the corresponding eigenfunctions for $\lambda_k$.

\begin{remark}
The largest eigenvalue of the corresponding normalized Laplacian can be characterized by
\begin{equation*}
         \lambda_N=\max_{f:V\rightarrow \mathbb{R}} \frac{\sum_{v\sim w}\biggl(f(v)-f(w)\biggr)^2}{\sum_{v\in V}\deg v\cdot f(v)^2}.
\end{equation*}
Furthermore, any function $f\in C(V)\setminus \{\boldsymbol{0}\}$ attaining this maximum is an eigenfunction of $L$ with eigenvalue $\lambda_N$.
\end{remark}

We shall now give the definitions of independent sets, twin vertices and duplicate vertices.
\begin{definition}\label{def:indepset}
    Let $U\subseteq V$. We say that $U$ is an \emph{independent set} if, for all pairs of vertices $u_1,u_2\in U$, we have that $u_1\not\sim u_2$.
\end{definition}

\begin{definition}\label{def:twinduplicate}
   Given $u\in V$, we let $N(u)$ denote the set of all \emph{neighbors} of $u$, i.e., the set of all vertices that are adjacent to $u$.
    If two distinct vertices $v,w\in V$ have the property that
    \[
    N(v)\setminus\{w\} = N(w)\setminus\{v\},
    \]
    then $v$ and $w$ are \emph{twin vertices} if $v\sim w$, while $v$ and $w$ are \emph{duplicate vertices} if $v\not\sim w$.
\end{definition} 
We refer to \cite{Butler2016} for an extensive study of twin vertices, duplicate vertices and twin subgraphs.

\begin{definition}\label{def:e(U,V)} Let $U_1,U_2\subset V$ be subsets of the vertex set of $G$. We let
    \[
    e\bigl(U_1,U_2\bigr)\coloneqq \bigl|\bigl\{ \{u,v\}\in E\colon u\in U_1,v\in U_2\bigr\}\bigr|.
    \]
    Moreover, if $U_1=\{v\}$ for some $v\in V$, we let
    \[
    e\bigl(v,U_2\bigr)\coloneqq e\bigl(\{v\},U_2\bigr).
    \]
\end{definition}

\begin{definition}\label{def:f_vw}
Let $v,w\in V$ be distinct vertices. We let
    \begin{equation*}
    f_{v,w}(u) \coloneqq \begin{cases}
        1, &\text{ if }u= v,\\
        -1, &\text{ if }u=w,\\
        0, &\text{ otherwise.}
    \end{cases}
    \end{equation*}
\end{definition}

\begin{definition}\label{def:f_ij}
   Fix a proper $k$-coloring of $G$ with coloring classes $V_1,\ldots,V_k$. Given two distinct indices $i,j\in\{1,\ldots,k\}$, we define $f_{ij}\colon V\to\R$ by
    \[
    f_{ij}(v) \coloneqq \begin{cases}
        1, &\text{ if } v\in V_i,\\
        -1, &\text{ if } v\in V_j,\\
        0, &\text{ otherwise.}
    \end{cases}
    \]
\end{definition}
Note that, if $G$ bipartite and $V_1$ and $V_2$ denote its bipartition classes, then the function $f_{12}$ from Definition \ref{def:f_ij} has the property that $$\RQ(f_{12}) = 2 = \lambda_N = \frac{\chi}{\chi-1}.$$ If $G$ is a complete graph, then any function $f_{ij}$ from Definition \ref{def:f_ij} is of the form $f_{v,w}$ as in Definition \ref{def:f_vw} for some $v,w\in V$, and it has the property that
    \[
    \RQ(f_{ij}) = \frac{N}{N-1} = \lambda_N = \frac{\chi}{\chi-1}.
    \]

    We conclude the section with the following definitions.

\begin{definition}\label{def:equitable}\cite{Gabriel-colouring}
    Let $M$ be an $n\times n$ matrix and let $\pi = \{S_1,\ldots,S_m\}$ be a partition of $\{1,\ldots,n\}$. The partition $\pi$ is \emph{equitable to $M$} if, for all $S_{i_1}\neq S_{i_2}$, and for all $j\in S_{i_1}$, the sum $\sum_{k\in S_{i_2}}M_{jk}$ is constant. \end{definition} 

    In particular, we are interested in equitable partitions in the case where $M=D^{-1}A$ or $M=A$. We let, for $k\geq \chi$, $V_1,\ldots,V_k$ be the coloring classes with respect to some fixed proper $k$-coloring $c$ of $V$.
    \begin{itemize}
        \item We say that $c$ is \emph{equitable with respect to $D^{-1}A$} if, for all $i=1,\ldots,k$ and all $v\in V$, we have that
        \[
        e(v,V_i) = \begin{cases}
        \frac{\deg v}{k-1}, &\text{ if } v\notin V_i,\\
        0, &\text{ if } v\in V_i.
        \end{cases}
        \]
        \item We say that $c$ is \emph{equitable with respect to $A$ if, for all $i,j=1,\ldots,k$}, for all $v_i,w_i \in V_i$, we have that $$e(v_i,V_j) = e(w_i,V_j).$$
    \end{itemize}

   % Let $V_1,\ldots,V_\chi$ be the coloring classes with respect to some fixed proper $\chi$-coloring $c$ of $V$. We say that $c$ is \emph{equitable with respect to $D^{-1}A$} if for all $i=1,\ldots,\chi$ and all $v\in V$, we have that
    %\[
    %e(v,V_i) = \begin{cases}
    %    \frac{\deg v}{\chi-1}, &\text{ if } v\notin V_i,\\
    %    0, &\text{ if } v\in V_i.
    %\end{cases}
    %\]

\subsection{Given families of graphs with corresponding coloring number and spectrum}\label{subsec:listofgphs}
We shall now list some special graphs, together with their coloring number and their spectrum with respect to the normalized Laplacian. A more elaborate list of graphs and their spectra, also including the spectra with respect to other matrices than the normalized Laplacian, can be found in \cite{Cavers2010}. We use the notation
\[
\bigl\{ \mu_1^{(m_1)},\ldots,\mu_p^{(m_p)} \bigr\}
\]
to denote a multiset which contains the element $\mu_i$ with multiplicity $m_i$. Throughout the paper, we shall also use the notation $m_G(\lambda)$ for the multiplicity of $\lambda$ as an eigenvalue of the normalized Laplacian of $G$.

\begin{enumerate}
    \item The complete graph $K_N$ on $N$ vertices has coloring number $\chi = N$ and spectrum
    \[
    \biggl\{ \frac N{N-1}^{(N-1)},0^{(1)}\biggr\}.
    \]
    \item The complete bipartite graph $K_{N_1,N_2}$ on $N=N_1+N_2$ vertices has coloring number $\chi =2$ and spectrum
    \[
    \biggl\{2^{(1)},1^{(N-2)},0^{(1)}\biggr\}.
    \]
    A special example is the star graph $S_N = K_{N-1,1}$.
    \item The complete multipartite graph with partition classes of the same size $K_{\underbrace{N_1,\ldots,N_1}_k}$ on $N=k\cdot N_1$ vertices, which can be equivalently described as the Turán graph $T(N,k)$ \cite{turan1941,sidorenko1995we,furedi1991}, has coloring number $\chi=k$ and spectrum
    \[
    \biggl\{\frac{k}{k-1}^{(k-1)},1^{(N-k)},0^{(1)}\biggr\}.
    \]
    \item The \emph{$m$-petal graph} \cite{JMZ23} on $N=2m+1$ vertices is the graph with vertex set
    $$V = \{x,v_1,\ldots,v_m,w_1,\ldots,w_m\}$$
    and edge set
    \[
     E = \bigcup_{i=1}^m \biggl\{ \{x,v_i\},\{x,w_i\},\{v_i,w_i\} \biggr\}.
    \]
    An example can be found in Figure \ref{fig:petalgph} below. Its coloring number $\chi$ equals $3$, and its spectrum is
    \[
    \biggl\{ \frac{3}{2}^{(m+1)}, \frac12^{(m-1)},0^{(1)} \biggr\}.
    \]
\end{enumerate}

\section{Graphs with largest eigenvalue \texorpdfstring{$\chi/(\chi-1)$}{chi/(chi-1)}}\label{section:chi/(chi-1)}
\subsection{Literature review}
 Also in this section we fix a connected simple graph $G=(V,E)$ on $N\geq 2$ vertices.

The following theorem gives a lower bound for the maximum eigenvalue of the normalized Laplacian of a graph in terms of its coloring number. It was first proven by Elphick and Wocjan (2015) \cite{ElphickWocjan2015} (Equation 20) as a consequence of Theorem 1 from Nikiforov (2007) \cite{Nikiforov2007}. Elphick and Wocjan also generalized the bound to include more normalized Laplacian eigenvalues (Equation 21).
Furthermore, the theorem was proven by
Coutinho, Grandsire and Passos (2019) \cite{Gabriel-colouring} (Lemma 6)
and by Sun and Das (2020) \cite{SunDas2020} (Theorem 3.1). A generalization for hypergraphs was proven by Abiad, Mulas and Zhang (2021) \cite{Abiad20} (Corollary 5.4).
\begin{theorem}\label{thm:3.1DS}
We have that
\begin{equation}\label{eq:ourbound}
\lambda_N\geq \frac{\chi}{\chi-1},
\end{equation} and this inequality is sharp.
\end{theorem}

Rewriting the inequality in \eqref{eq:ourbound} gives a lower bound for the chromatic number,
\[
\chi \geq \frac{\lambda_N}{\lambda_N-1}.
\]
Moreover, for regular graphs, the bound from Equation \ref{eq:ourbound} coincides with the Hoffman bound \cite{Hoffman1970},
\begin{equation}\label{eq:hoffmanbound}
    \chi \geq 1 - \frac{\mu_1}{\mu_N},
\end{equation}
where $\mu_1$ and $\mu_N$ denote the smallest and largest eigenvalues of the adjacency matrix, respectively. Graphs for which the Hoffman bound is sharp have been studied, for example, in \cite{HaemersTonchev1996,FialaHaemers20065-chromatic,Godsiletal2016,Roberson2019SRGs,ThijsThesis}.

In \cite{SunDas2020}, Sun and Das state the following analogous open question:
\begin{question}\label{qu:mainqu}
Which connected finite graphs satisfy $\lambda_N=\chi/(\chi-1)$?
\end{question}
Sun and Das also give a couple of graphs for which this equality holds, including complete multipartite graphs with partition classes of equal size, and $m$-petal graphs (cf.\ Section \ref{subsec:listofgphs} and \cite{JMZ23}).
Furthermore, Coutinho, Grandsire and Passos (2019) \cite{Gabriel-colouring} prove the following  necessary property for graphs with largest eigenvalue $\chi/(\chi-1)$.

\begin{theorem}[Coutinho,  Grandsire \& Passos (2019) \cite{Gabriel-colouring}, Theorem 7]\label{thm:edgespread}
    If $\lambda_N=\chi/\bigl(\chi-1\bigr)$, then every proper $\chi$-coloring of $G$ is equitable with respect to $D^{-1}A$, i.e., for a fixed proper $\chi$-coloring of $G$ with coloring classes $V_1,\ldots,V_\chi$, we have for all $1\leq i\leq \chi$ and all $v\in V$ that
    \[
    e(v,V_i) = \begin{cases}
        \frac{\deg v}{\chi-1}, &\text{ if } v\notin V_i,\\
        0, &\text{ if }v\in V_i.
    \end{cases}
    \]
\end{theorem}\color{black}
We offer an alternative proof to Theorem \ref{thm:edgespread}.
\begin{proof}
    Without loss of generality, we may assume that
    \begin{align}
    e\bigl(V_1,V_2\bigr) &= \max_{1\leq i < j \leq \chi}e\bigl(V_i,V_j\bigr).\label{eq:e(V1,V2)}
    \end{align}
    Let $W\coloneqq V\setminus \bigl(V_1\cup V_2\bigr)$, and 
 consider the function $f_{12}$ from Definition \ref{def:f_ij}.
    We have that
    \begin{align}
    \begin{split}
        \RQ\bigl(f_{12}\bigr)
        &= \frac{\sum_{v\sim w}\bigl(f_{12}(v)-f_{12}(w)\bigr)^2}{\sum_{v\in V}\deg v\cdot f_{12}(v)^2}\\
        &= 1 - 2\frac{\sum_{v\sim w}f_{12}(v)\cdot f_{12}(w)}{\sum_{v\in V}\deg v\cdot f_{12}(v)^2}\\
        &= 1 + \frac{2\cdot e(V_1,V_2)}{2e\bigl(V_1,V_2\bigr)+e\bigl(V_2,W\bigr)+e\bigl(V_1,W\bigr)}\\
        &\geq 1 + \frac{2e(V_1,V_2)}{2e(V_1,V_2)+2(\chi-2)e(V_1,V_2)}\label{eq:inequality}\\
        &= \frac{\chi}{\chi-1}.
    \end{split}
    \end{align}
    Together with the assumption that $\lambda_N=\chi/(\chi-1)$, this implies that $$\RQ(f_{12})=\frac{\chi}{\chi-1}.$$ Hence, the inequality \eqref{eq:inequality} must be an equality, implying that
    \[
    e(V_1,V_2) = e(V_1,V_i), \quad \text{for $i=3,\ldots,\chi$.}
    \]
    We can thus, for $i=3,\ldots,\chi$, calculate $\RQ(f_{1i})$ analogously to $\RQ(f_{12})$, to see that
    \[
    \RQ(f_{1i}) = \frac{\chi}{\chi-1} = \lambda_N, \quad \text{for $i=2,\ldots,\chi$.}
    \]
    As a consequence, we have that
    \[
    e(V_1,V_2) = e(V_1,V_i) = e(V_i,V_j), \quad \text{for $1\leq j\leq \chi$ such that $j\neq i$.}
    \]
     We can use this to see that
    \[
    \RQ(f_{ij}) = \frac{\chi}{\chi-1} = \lambda_N, \quad \text{for $1\leq i<j\leq \chi$.}
    \]
    Now let $v\in V(G)$, and let $i$ be such that $v\in V_i$. Then, by definition of proper $\chi$-coloring, we have that $e(v,V_i) = 0$. Now consider $j\neq i$ such that $1\leq j\leq\chi$. 
    By the min-max Principle \eqref{min-max}, we have that $(\chi/(\chi-1),f_{ij})$ is an eigenpair of $L$. Equation \eqref{eq:eigenpair} gives us that
    \begin{equation*}
        \frac{1}{\chi-1}=\frac{1}{\deg v}\cdot e(v,V_j).
    \end{equation*} 
    By rewriting this, we conclude that
    \[
    e(v,V_i) = \begin{cases}
        \frac{\deg v}{\chi-1} &\text{ if } v\notin V_i,\\
        0 &\text{ if } v\in V_i.
    \end{cases} \qedhere
    \]
\end{proof}
Note that we can use our estimate of $\RQ(f_{12})$ from the proof of Theorem \ref{thm:edgespread} in combination with the min-max Principle \eqref{min-max}, to see that, for any graph with coloring number $\chi$, we have that
\[
\lambda_N\geq \RQ(f_{12})\geq \frac{\chi}{\chi-1},
\]
which gives an alternative proof of Theorem \ref{thm:3.1DS}. Furthermore, we see that $\lambda_N = \chi/(\chi-1)$ if and only if the function $f_{12}$, with the assumption from Equation \eqref{eq:e(V1,V2)}, maximizes the Rayleigh quotient.

In \cite{Abiad2019weight-regular}, Abiad gave a necessary condition for graphs for which the Hoffman bound is sharp. Such condition coincides with the one in Theorem \ref{thm:edgespread} for regular graphs.

Coutinho, Grandsire and Passos (2019) \cite{Gabriel-colouring} also prove the following in their proof of Theorem 13.
\begin{proposition}[Coutinho, Grandsire and Passos (2019) \cite{Gabriel-colouring}]\label{prop:multchi-1}
    If $\lambda_N = \chi/(\chi-1)$, then the multiplicity of the eigenvalue $\chi/(\chi-1)$ is at least $\chi-1$.
\end{proposition}
Blokhuis, Brouwers and Haemers (2007, Proposition 2.3) \cite{BlokhuisBrouwerHaemers2007} proved that the same is true for $k$-regular graphs for which the Hoffman bound is sharp, which is a special case of Proposition \ref{prop:multchi-1}. They also prove that, if the multiplicity of the smallest eigenvalue of the adjacency matrix equals $\chi-1$, then the graph admits a unique $\chi$-coloring.
This result can be generalized for the bound from Equation \eqref{eq:ourbound} to apply to all graphs.
\begin{proposition}\label{prop:multlambdaN}
    If $\lambda_N=\chi/(\chi-1)$ and $\chi/(\chi-1)$ has multiplicity equal to $\chi-1$, then $G$ admits only one proper $\chi$-coloring, up to a permutation of the coloring classes.
\end{proposition}
\begin{proof}
    The proof is the same as in Proposition 2.3 from \cite{BlokhuisBrouwerHaemers2007}. We fix a proper $\chi$-coloring $c$ with coloring classes $V_1,\ldots,V_\chi$. Then, for $2\leq j\leq \chi$, the functions $f_{1j}$ from Definition \ref{def:f_ij} form $\chi-1$ linearly independent eigenfunctions with eigenvalue $\chi/(\chi-1)$. If there is a proper $\chi$-coloring whose coloring classes are not a permutation of the coloring classes of $c$, then we find an eigenfunction that is not in the span of $\{f_{1j}, 2\leq j\leq \chi\}$, contradicting our assumption.
\end{proof}
In Section \ref{section:counter} we shall see that the opposite implication of Proposition \ref{prop:multlambdaN} is not true.

Another concept which is relevant for graphs for which the bound in Equation \eqref{eq:ourbound} is sharp, is the concept of twin and duplicate vertices from Definition \ref{def:twinduplicate}. Butler (2016) \cite{Butler2016} studied twins and duplicates, and the more general concept of twin subgraphs.

Note that we have the following spectral characterization of twins and duplicates.

\begin{lemma}[Special case of Theorem 4 in \cite{Butler2016}.]\label{lem:Butlerspectral}
    Let $v,w\in V$ be distinct vertices. Then, $v$ and $w$ are twins or duplicates if and only if the function $f_{v,w}$ from Definition \ref{def:f_vw} is an eigenfunction. In this case, its eigenvalue equals $1$ if $v$ and $w$ are duplicates, and $(\deg v+1)/\deg v$ if $v$ and $w$ are twins.
\end{lemma}
Butler (2016) \cite{Butler2016} used this characterization to prove the following.
\begin{proposition}[Corollary $1$ in \cite{Butler2016}]\label{prop:Butlertwins}
\begin{enumerate}
    \item Let $D_i$ consist of a collection of duplicate vertices. Then there are $|D_i|-1$ eigenvalues of $1$ which come from eigenvectors restricted to $D_i$.
    \item Let $T_i$ consist of a collection of twin vertices which have common degree $d$. Then there are $|T_i|-1$ eigenvalues of $(d+1)/d$ which come from eigenvectors restricted to $T_i$.
    \end{enumerate}
\end{proposition}

\subsection{Graphs with largest eigenvalue \texorpdfstring{$\chi/(\chi-1)$}{chi/(chi-1)}}

We now prove some results for graphs that admit a coloring which is equitable with respect to $D^{-1}A$. As we have already seen, such graphs are relevant for Question \ref{qu:mainqu}, because for graphs which satisfy $\lambda_N = \chi/(\chi-1)$, all proper $\chi$-colorings are equitable with respect to $D^{-1}A$.

We shall first prove some results about the eigenfunctions of graphs that admit a coloring which is equitable with respect to $D^{-1}A$. Then, we shall prove some results regarding twin and duplicate vertices of graphs admitting equitable colorings with respect to $D^{-1}A$. Finally, we shall prove the main result of this section, namely Corollary \ref{cor:decomposition}, which tells us that, given a graph with largest eigenvalue $\chi/(\chi-1)$, we can remove some coloring classes of any proper $\chi$-coloring to obtain a new graph for which the bound from Equation \eqref{eq:ourbound} is also sharp.

The following result generalizes a result from the proof of Proposition 2.3 from Blokhuis, Brouwers and Haemers (2007) \cite{BlokhuisBrouwerHaemers2007}.
\begin{proposition}\label{prop:eigenfns}
Let $k\geq \chi$.
Fix a proper $k$-coloring $c$ and let $V_1,\ldots,V_k$ denote the corresponding coloring classes. If $c$ is equitable with respect to $D^{-1}A$,
then, for all $i\neq j$ such that $1\leq i,j\leq k$, the function $f_{ij}$ from Definition \ref{def:f_ij} is an eigenfunction of $L$ with corresponding eigenvalue $k/(k-1)$.
\end{proposition}
\begin{proof}
By \eqref{eq:eigenpair}, $(k/(k-1),f_{ij})$ is an eigenpair for $L$ if and only if,  for all $v\in V$,
\begin{equation*}
    -\frac{1}{k-1} f_{ij}(v)=\frac{1}{\deg v}\sum_{w\sim v}f_{ij}(w).
\end{equation*}
Hence, by fixing $v_i\in V_i$, $v_j\in V_j$ and $v_0\in V\setminus \bigl(V_i\cup V_j\bigr)$, we obtain that
\begin{align*}
    -\frac1{k-1}f\bigl(v_i\bigr) &= -\frac{e\bigl(v_i,V_j\bigr)}{\deg v_i} = \frac1{\deg v_i}\sum_{w\sim v_i}f_{ij}(w);\\
    -\frac1{k-1}f\bigl(v_j\bigr) &= \frac{e\bigl(v_j,V_i\bigr)}{\deg v_j} = \frac1{\deg v_j}\sum_{w\sim v_j}f_{ij}(w);\\
    -\frac1{k-1}f\bigl(v_0\bigr) &= \frac{e\bigl(v_0,V_i\bigr)-e\bigl(v_0,V_j\bigr)}{\deg v_0} = \frac1{\deg v_0}\sum_{w\sim v_0}f_{ij}(w).
\end{align*}
We conclude that $(k/(k-1),f_{ij})$ is an eigenpair.
\end{proof}
\begin{remark}
One can use Proposition \ref{prop:eigenfns} in combination with Theorem \ref{thm:edgespread} to prove Proposition \ref{prop:multchi-1}, analogously to the proof by Blokhuis, Brouwers and Haemers (2007) \cite{BlokhuisBrouwerHaemers2007} for regular graphs.
\end{remark}

The following corollary of Proposition \ref{prop:eigenfns} concerns graphs that admit a proper $k$-coloring that is equitable with respect to $D^{-1}A$.
\begin{corollary}\label{cor:orthog}
    Let $k\geq \chi$.
    Assume that there exists a proper $k$-coloring $c$ with coloring classes $V_1,\ldots,V_k$ which is equitable with respect to $D^{-1}A$.
    Let $f$ be an eigenfunction corresponding to a non-zero eigenvalue $\lambda\neq k/(k-1)$. For all $j\in\{1,\ldots,k\}$ we have that
     \[
    \sum_{v\in V_j}\deg v f(v) = 0.
    \]
\end{corollary}
\begin{proof}
   Since $\lambda\neq k/(k-1)$, for each $j\in\{2,\ldots,k\}$ we have that  $f$ and $f_{1j}$ are orthogonal, implying that
    \begin{equation}\label{eq:f_{ij}}
    \sum_{v\in V_1} \deg v f(v) = \sum_{v\in V_j}\deg v f(v).
    \end{equation}
   Moreover, since the constant functions are the eigenfunctions of the eigenvalue $0$, we also have that $f$ is orthogonal to the constant functions, implying that
    \begin{equation}\label{eq:eigenv0}
    \sum_{v\in V}\deg v f(v) = 0.
    \end{equation}
    Combining \eqref{eq:f_{ij}}  and \eqref{eq:eigenv0}  proves the claim.
\end{proof}

We shall now prove a proposition about graphs with largest eigenvalue $\chi/(\chi-1)$ regarding their twin and duplicate vertices (cf.\ Definition \ref{def:twinduplicate}), which follows from Theorem \ref{thm:edgespread} and Lemma \ref{lem:Butlerspectral}, both of which are due to Butler (2016) \cite{Butler2016}.
\begin{proposition}\label{prop:twinduplicate}
    Assume that $\lambda_N=\chi/\bigl(\chi-1\bigr)$.
    
    \begin{enumerate}[label=(\roman*)]
    \item If $v_1\sim v_2$ are twin vertices, then 
    \[
    \deg v_1 = \deg v_2 = \chi-1.
    \]
    Moreover, the function $f_{v_1,v_2}$ from Definition \ref{def:f_vw} is an eigenfunction with eigenvalue $\chi/(\chi-1)$.
    \item If $w_1\not\sim w_2$ are duplicate vertices, then for any proper $\chi$-coloring, $w_1$ and $w_2$ must be in the same coloring class.
    \end{enumerate}
    
\end{proposition}
\begin{proof}
\begin{enumerate}[label=(\roman*)]
    \item
    Let $V_1,\ldots,V_\chi$ be the coloring classes of $G$ with respect to a fixed proper $\chi$-coloring. Without loss of generality, we may assume that $v_1\in V_1$ and $v_2\in V_2$. Since $v_1$ and $v_2$ are twin vertices, we must have
    \[
    e(v_1,V_2) = e(v_2,V_1)=1,
    \]
    from which it follows by Theorem \ref{thm:edgespread} that
    \[
    \deg v_1 = \deg v_2 = \chi-1.
    \]
    Moreover, since one can easily check that $\RQ(f_{v_1,w_1})=\chi/(\chi-1)$, we have that  the function $f_{v_1,v_2}$ is an eigenfunction with eigenvalue $\chi/(\chi-1)$. 
    
    \item
    Let $c$ be a proper $\chi$-coloring and assume, by contradiction, that $w_1$ and $w_2$ are not in the same coloring class. Then, the proper $\chi$-coloring $c'$ defined by
    \[
    c'(v) \coloneqq \begin{cases}
        c(w_1) &\text{ if } v = w_2,\\
        c(v) &\text{ otherwise,}
    \end{cases}
    \]
    does not satisfy the statement of Theorem \ref{thm:edgespread}, which is a contradiction.\qedhere
    
\end{enumerate}
\end{proof}

We can use Proposition \ref{prop:Butlertwins} from Butler (2016) \cite{Butler2016} to find upper and lower bounds for the multiplicity of $\chi/(\chi-1)$.
\begin{proposition}\label{prop:multiplicity}
    Assume that $\lambda_N=\chi/\bigl(\chi-1\bigr)$, and fix a proper $\chi$-coloring with coloring classes $V_1,V_2,\ldots, V_\chi$. For given $j\geq0$ and $k\geq 0$, let
    \[
    D_1,\ldots,D_j\subseteq V(G) \quad \text{ and } \quad T_1,\ldots,T_k\subseteq V(G) 
    \]
    be mutually disjoint sets, where $D_1,\ldots,D_j$ form a collection of duplicate vertices, while $T_1,\ldots,T_k$ form a collection of twin vertices. Additionally, fix the smallest $y$ such that
    \[
    \bigcup_{i=1}^jT_i \subseteq V_1\cup\cdots \cup V_y,
    \]
    where we change the order of the coloring classes $V_i$ if necessary, and we let $y=0$ if $k=0$.
    Then, 
    \begin{equation}\label{eq:mult-lambdaN}
    \sum_{i=1}^j|T_i|-j+\chi-y\leq m_G\left(\frac{\chi}{\chi-1}\right)\leq N - \sum_{i=1}^k|D_i| + k - 1.
     \end{equation}
    Furthermore, the upper bound is tight if and only if $\bigcup_{i=1}^kD_i=V(G)$ and $G$ is a complete multipartite graph with partition classes of equal size.
\end{proposition}
\begin{proof}
    The proof of the inequalities is straightforward, using Proposition \ref{prop:Butlertwins} and Proposition \ref{prop:eigenfns}. Furthermore, one can check that the upper bound is tight if and only if $G$ has spectrum
    \[
    \biggl\{ 0^{(1)},1^{(N-\chi)},\frac{\chi}{\chi-1}^{(\chi-1)}\biggr\},
    \]
    and all eigenfunctions with eigenvalue $1$ must come from pairs of duplicate vertices. This is only possible if $G$ is a complete multipartite graph with $\chi$ partition classes of the same size $N/\chi$.
\end{proof}

Examples of graphs for which the lower bound from Proposition \ref{prop:multiplicity} is tight include complete multipartite graphs with partition classes of the same size, complete graphs, bipartite graphs, and petal graphs (cf.\ Section \ref{subsec:listofgphs} and \cite{JMZ23}).

We now prove the following theorem, concerning the support of functions $f\colon V\to \R$.
\begin{theorem}\label{thm:support}
    Let $G$ be a graph with coloring number $\chi$, and let $V_1,\ldots,V_\chi$ be the coloring classes with respect to a fixed proper $\chi$-coloring $c$. Assume that $c$ is equitable with respect to $D^{-1}A$.
    Moreover, let $f\colon V\to \R$ be a function such that $\supp(f)\subseteq \bigcup_{i\in I}V_i$, for some set $I\subseteq \{1,\ldots,\chi\}$ with $|I|\geq 1$. Then, we have that
    \[
    \RQ(f) = \frac{|I|-1}{\chi-1}\RQ_{G\left(\bigcup_{i\in I}V_i\right)}\left(f_{G\left(\bigcup_{i\in I}V_i\right)}\right) +\frac{\chi-|I|}{\chi-1}.
    \]
\end{theorem}
\begin{proof}
    The claim follows by the following calculation, in which we let $W_1\coloneqq \bigcup_{i\in I}V_i$ and $W_2\coloneqq V\setminus W_1$.
    \begin{align*}
        \RQ_G(f) &= \frac{\sum_{\substack{v,w\in W_1\\v\sim w}}\left(f(v)-f(w)\right)^2}{\sum_{v\in W_1}f(v)^2\deg v} + \frac{\sum_{\substack{v\in W_1\\w\in W_2\\v\sim w}}f(v)^2}{\sum_{v\in W_1}f(v)^2\deg v}\\
        &= \frac{|I|-1}{\chi-1}\frac{\sum_{\substack{v\in V_1\\w\in V_2\\v\sim w}}\left(f(v)-f(w)\right)^2}{\sum_{v\in W_1}f(v)^2\frac{(|I|-1)\deg v}{\chi-1}} + \frac{\sum_{v\in W_1}\frac{(\chi-|I|)\deg v}{\chi-1}f(v)^2}{\sum_{v\in W_1}f(v)^2\deg v}\\
        &= \frac{|I|-1}{\chi-1}\RQ_{G(W_1)}\left(f_{G(W_1)}\right) +\frac{\chi-|I|}{\chi-1}.\qedhere
    \end{align*}
\end{proof}
Theorem \ref{thm:support} has the following immediate corollary.
\begin{corollary}
    Let $G$ be a graph with coloring number $\chi$, and let $V_1,\ldots,V_\chi$ be the coloring classes with respect to a fixed proper $\chi$-coloring $c$. Assume that $c$ is equitable with respect to $D^{-1}A$.
    Moreover, let $f\colon V\to \R$ be a function such that $\supp(f)\subseteq \bigcup_{i\in I}V_i$, for some set $I\subseteq \{1,\ldots,\chi\}$ with $|I|\geq 1$. Then, we have that
    \[
    \RQ(f) \leq \frac{(|I|-1)\lambda_{\max}\bigl(G\left(\bigcup_{i\in I}V_i\right)\bigr) + \chi-|I|}{\chi-1}.
    \]
\end{corollary}

Theorem \ref{thm:support} also has the following corollary. A version of this corollary for the Hoffman bound was first proposed by Van Veluw (2024) \cite{ThijsThesis}.

\begin{corollary}\label{cor:decomposition}
    Let $G$ be a graph with coloring number $\chi\geq2$ such that $\lambda_{\max}(G) = \chi/(\chi-1)$. Let $V_1,\ldots,V_\chi$ be the coloring classes with respect to some fixed proper $\chi$-coloring. Then, for any subset $I\subseteq\{1,\ldots,\chi\}$ with $|I|\geq 2$,
    we have that
    \[
    \lambda_{\max}\biggl(G\left(\bigcup_{i\in I}V_i\right)\biggr) = \frac{|I|}{|I|-1} = \frac{\chi\biggl(G\left(\bigcup_{i\in I}V_i\right)\biggr)}{\chi\biggl(G\left(\bigcup_{i\in I}V_i\right)\biggr)-1}.
    \]
\end{corollary}
\begin{proof}
    Let $1\leq i\leq \chi$ and let $G_I\coloneqq G\left(\bigcup_{i\in I}V_i\right)$. Let $f_I$ be such that $\RQ(f_I) = \lambda_{\max}\bigl(G_I\bigr)$, and define $f\colon V(G)\to\R$ by
    \[
    f(v)\coloneqq \begin{cases}
        f_I(v), &\text{ if } v\in V\bigl(G_I\bigr),\\
        0, &\text{ otherwise.}
    \end{cases}
    \]
    Then, by Theorem \ref{thm:support} we have that
    \begin{align*}
        \frac{\chi}{\chi-1} &\geq \RQ_G(f) \\
        &= \frac{|I|-1}{\chi-1}\cdot\RQ_{G_I}(f_I)+\frac{\chi-|I|}{\chi-1}\\
        &=  \frac{|I|-1}{\chi-1}\cdot\lambda_{\max}(G_I) +\frac{\chi-|I|}{\chi-1}.
    \end{align*}
    This implies that
    \begin{align*}
        \chi &\geq \bigl(|I|-1\bigr)\cdot\lambda_{\max}\bigl(G_I\bigr) + \chi-|I|,
    \end{align*}
    which is equivalent to
    \[
    \lambda_{\max}\bigl(G_I\bigr) \leq \frac{|I|}{|I|-1} = \frac{\chi\bigl(G_I\bigr)}{\chi\bigl(G_I\bigr)-1}.
    \]
    Since we also have the opposite inequality by the bound \eqref{eq:ourbound}, we can conclude that
    \[
    \lambda_{\max}(G_I) =\frac{\chi(G_I)}{\chi(G_I)-1}.\qedhere
    \]
\end{proof}
\begin{remark}
    Assume that we are in the setting of Corollary \ref{cor:decomposition} and its proof.
    Then, the function $f$ in the proof is an eigenfunction of $G$ with eigenvalue $\chi/(\chi-1)$. Hence, we have that
    \[
    m_{G}\biggl(\frac{\chi}{\chi-1}\biggr) \geq m_{G_I}\biggl(\frac{|I|}{|I|-1}\biggr) + \chi - |I|.
    \]
\end{remark}
We also have the following result about eigenfunctions of graphs that admit an equitable $\chi$-coloring.
\begin{proposition}
    Let $G$ be a graph with coloring number $\chi$, and let, for some $k\geq \chi$, $V_1,\ldots,V_k$ be the coloring classes with respect to a fixed proper $k$-coloring $c$. Assume that $c$ is equitable with respect to $D^{-1}A$.
    Moreover, let $f\colon V\to \R$ be an eigenfunction with corresponding eigenvalue $\lambda$, such that $\supp(f)\subseteq \bigcup_{i\in I}V_i$ for some set $I\subseteq \{1,\ldots,k\}$ with $|I|\geq 2$. Then, we have that $f|_{G\left(\bigcup_{i\in I}V_i\right)}$ is an eigenfunction of $G\left(\bigcup_{i\in I}V_i\right)$ with eigenvalue $
    1+(k-1)(\lambda-1)/(|I|-1).
    $
\end{proposition}
\begin{proof}
    We use the notations $f_I\coloneqq f|_{G\left(\bigcup_{i\in I}V_i\right)}$ and $G_I\coloneqq G\left(\bigcup_{i\in I}V_i\right)$. Let $v\in \bigcup_{i\in I}V_i$. Then, it follows from Equation \eqref{eq:eigenpair}, that
    \begin{align*}
        \biggl(1-\biggl(1+\frac{(k-1)(\lambda-1)}{|I|-1}\biggr)\biggr)f_I(v) &= \biggl(\frac{k-1}{|I|-1}(1-\lambda)\biggr)f(v)\\
        &= \frac{k-1}{|I|-1}\cdot\frac1{\deg_G v}\sum_{\substack{w\in V(G):\\ \{w,v\}\in E(G)}}f(w)\\
        &= \frac1{\deg_{G_I}v}\sum_{\substack{w\in V(G_I):\\\{w,v\}\in E(G_I)}}f_I(w).
    \end{align*}
   We conclude that $f_I$ is an eigenfunction of $G_I$ with eigenvalue $1+(k-1)(\lambda-1)/(|I|-1)$.
\end{proof}

\section{Graphs with equal edge spread}\label{section:counter}
As before, we fix a connected simple graph $G=(V,E)$ on $N\geq 2$ vertices throughout the section.
In view of Theorem \ref{thm:edgespread}, it is natural to ask the following question.
\begin{question}\label{qu:equality}
    Is it true that $\lambda_N=\chi/(\chi-1)$ if and only if, for every proper $\chi$-coloring, its coloring classes $V_1,\ldots,V_\chi$ are such that
    \[
    e\bigl(v,V_i\bigr) = \begin{cases}
        \frac{\deg v}{\chi-1}&\text{ if } v\notin V_i,\\
        0 &\text{ if }v\in V_i?
    \end{cases}
    \]
\end{question}
By Theorem \ref{thm:edgespread}, the implication $(\Rightarrow)$ from Question \ref{qu:equality} is true. This section is dedicated to showing that the other implication does not hold, hence the answer to the question is \emph{no}.

In this section, we shall construct a family of graphs for which some members do not satisfy Question \ref{qu:equality}, as well as the opposite implication in Proposition \ref{prop:multlambdaN}. As a preliminary result, we first need to prove the following generalization of Lemma \ref{lem:Butlerspectral}, Proposition \ref{prop:eigenfns} and Proposition \ref{prop:eigenfns}.

\begin{proposition}\label{prop:plusminus}
    Let $V_+,V_-\subset V$ be disjoint subsets of $V$, and consider the function $f_{+-}$ defined by
    \[
    f_{+-}(v) := \begin{cases}
        1 &\text{ if } v\in V_+,\\
        -1 &\text{ if } v\in V_-,\\
        0&\text{ otherwise}
    \end{cases}
    \]    
    Then, $f_{+-}$ is an eigenfunction of $L$ with corresponding eigenvalue $\lambda$ if and only if the following two statements hold:
    \begin{enumerate}
        \item For all $v_0\notin V_+\cup V_-$, 
        \[
        e(v_0,V_+) = e(v_0,V_-).
        \]
        \item For all $v_-\in V_-$ and $v_+\in V_+$,
        \[
        \lambda-1 = \frac{e\bigl(v_-,V_+\bigr)-e\bigl(v_-,V_-\bigr)}{\deg v_-} = \frac{e\bigl(v_+,V_-\bigr)-e\bigl(v_+,V_+\bigr)}{\deg v_+}.
        \] 
         In particular, if $V_-$ and $V_+$ are independent sets,  then the above equation simplifies to
        \[
       \lambda-1 = \frac{e\bigl(v_-,V_+\bigr)}{\deg v_-} = \frac{e\bigl(v_+,V_-\bigr)}{\deg v_+}.
        \]
    \end{enumerate}
\end{proposition}
\begin{proof}
     By \eqref{eq:eigenpair}, $(f_{+-},\lambda)$ is an eigenpair if and only if, for all $v\in V$,
    \[
    (\lambda-1)f(v) = -\frac1{\deg v}\sum_{w\sim v}f(w).
    \]
    Hence, given $v_0\notin V_+\cup V_-$, $v_-\in V_-$ and $v_+\in V_+$, we have that $(f_{+-},\lambda)$ is an eigenpair if and only if the following three equations hold:
    
    \begin{align*}
        0 &= (\lambda-1)f(v_0) = \frac{e\bigl(v_0,V_+\bigr) - e\bigl(v_0,V_-\bigr)}{\deg v_0},\\
        \lambda-1 &= -(\lambda-1)f\bigl(v_-\bigr) = \frac{e\bigl(v_-,V_+\bigr) - e\bigl(v_-,V_-\bigr)}{\deg v_-},\\
        \lambda-1 &= (\lambda-1)f\bigl(v_+\bigr) = \frac{e\bigl(v_+,V_-\bigr) - e\bigl(v_+,V_+\bigr)}{\deg v_+}.\qedhere
    \end{align*}
\end{proof}
We dedicate the rest of this section to the construction and the study of a special family of graphs, with the aim of giving a counterexample to  Question \ref{qu:equality}. These graphs are constructed by taking a complete multipartite graph with $\theta$ partition classes of the same size, and removing disjoint $\theta$-cliques.

\begin{definition}\label{def:G_{k,theta}}
    Let $G_{k,\theta}^d$ with $k,\theta,d\geq 0$ and $0\leq d\leq k$ be the graph with vertex set
    \[
    V\bigl(G_{k,\theta}^d\bigr) = \bigcup_{i=1}^{\theta} \bigl\{v_1^i,\ldots,v_k^i\bigr\},
    \]
    where $v_{j_1}^{i_1}$ and $v_{j_2}^{i_2}$ are \emph{not} adjacent if and only if exactly one of the following holds:
    \begin{itemize}
        \item either $i_1 = i_2$, or
        \item $i_1\neq i_2$ and $j_1=j_2\leq d$.
    \end{itemize}
\end{definition}
Hence, $G_{k,\theta}^0$ is the complete multipartite graph that has $\theta$ coloring classes of size $k$, and we know from Section \ref{subsec:listofgphs} that it has spectrum
\[
\biggl\{ \frac{\theta}{\theta-1}^{(\theta-1)},1^{((k-1)\theta)},0^{(1)}\biggr\}.
\]
More generally, $G_{k,\theta}^d$ is given by the complete multipartite graph with $\theta$ partition classes
\[
V_i\coloneqq \bigl\{ v_1^i,\ldots,v_k^i\bigr\}
\]  of size $k$, in which $d$ disjoint $\theta$-cliques of edges are removed. Two examples of this graph are shown in Figure \ref{fig:G_{k,theta}^d}.

\begin{figure}
    \centering
    \subfigure[$G_{2,3}^1$]{\includegraphics[width=60mm]{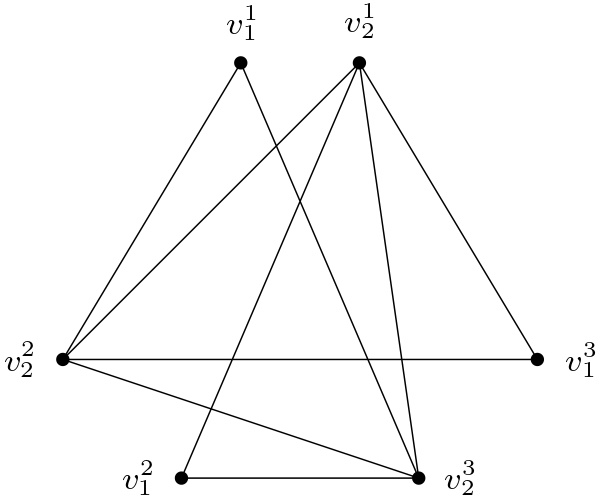}}
    \hfill
    \subfigure[$G_{4,3}^2$]{\includegraphics[width=80mm]{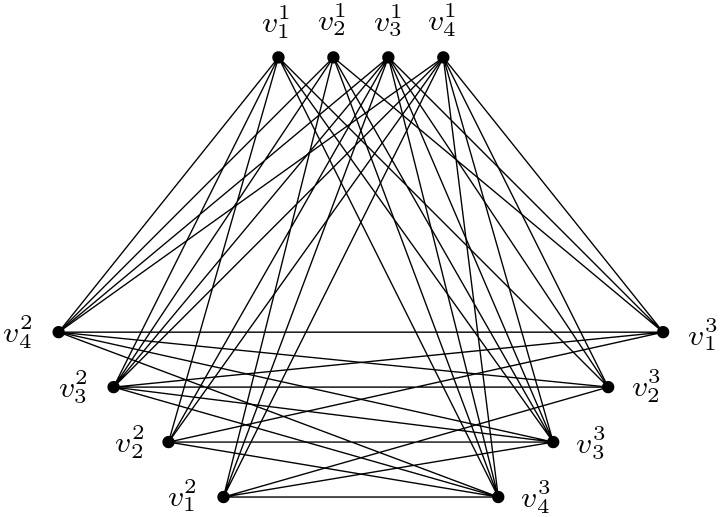}}
    \caption{Two examples of the graph $G_{k,\theta}^d$}
    \label{fig:G_{k,theta}^d}
\end{figure}

The following proposition and its corollary show that the answer to Question \ref{qu:equality} is \emph{no}.

\begin{proposition}\label{prop:G_{k,chi}}
    Let $k,\theta>1$ and fix $d$ such that $0\leq d\leq k$. 
    \begin{enumerate}[label=(\roman*)]
        \item If $d=k$, then $G_{k,\theta}^k$ is isomorphic to $G_{\theta,k}^\theta$.
        \item If either $d< k$, or $d=k$ and $k\geq \theta$, then the graph $G_{k,\theta}^d$ has coloring number $\theta$. 
        \item If $d<k$, then  $G_{k,\theta}^d$ has exactly one proper $\theta$-coloring, up to a permutation of the coloring classes. This is given by $c\colon V\to \{1,\ldots,\theta\}$ such that $c(v)=i \iff v\in V_i$.
        \item The graph $G_{k,\theta}^d$ satisfies
        \[
        e\bigl(v,V_i\bigr) = \begin{cases}
        \frac{\deg v}{\theta-1}&\text{ if } v\notin V_i,\\
        0 &\text{ if }v\in V_i.
        \end{cases}
        \]
        \item If $0< d<k$, then $G_{k,\theta}^d$ has spectrum
        \begin{align*}
        \Bigg\{ &\frac k{k-1}^{\bigl(d-1\bigr)},\frac{k^2-d}{k(k-1)}^{(1)},
        \frac{\theta}{\theta-1}^{\bigl(\theta-1\bigr)},  1^{\bigl((k-d-1)\theta\bigr)},\\ &\biggl(1-\frac{k-d}{k(k-1)(\theta-1)}\biggr)^{(\theta-1)}, \biggl(1-\frac1{(k-1)(\theta-1)}\biggr)^{(d-1)(\theta-1)}, 0^{(1)} \Bigg\}.
        \end{align*}
        \item If $d=k\geq \theta$ and $k\theta>4$, then $G_{k,\theta}^d$ has spectrum
         \begin{align*}
        \Bigg\{ &\frac{\theta}{\theta-1}^{\bigl(\theta-1\bigr)}, \frac k{k-1}^{\bigl(d-1\bigr)}, \biggl(1-\frac1{(k-1)(\theta-1)}\biggr)^{(d-1)(\theta-1)}, 0^{(1)} \Bigg\}.
        \end{align*}
    \end{enumerate}
\end{proposition}

\begin{proof}
    \begin{enumerate}[label=(\roman*)]
    \item It is easy to check that an isomorphism is given by 
    \begin{align*}
        V\Bigl(G_{k,\theta}^k\Bigr) &\to V\Bigl(G_{\theta,k}^{\theta}\Bigr)\\
        v_j^i &\mapsto v_i^j.
    \end{align*}

    \item If $d< k$, then the coloring number of $G_{k,\theta}^d$ equals $\theta$ because the graph contains the $\theta$-clique $\{v_{d+1}^1,\ldots,v_{d+1}^\theta\}$. Similarly, if $d=k$ and $k\geq \theta$, then the graph $G_{k,\theta}^d$ has coloring number $\theta$ since it contains the $\theta$-clique $\{v_1^1,\ldots,v_\theta^\theta\}$.

    \item Assume that $d<k$. Note that, for any proper coloring, for all $i$ such that $1\leq i\leq\theta$ we must have that the vertices $v_{d+1}^i$ have different colors, since they form a $\theta$-clique. Now let $c$ be a proper $\theta$-coloring
    such that $c\bigl(v_{d+1}^i\bigr) = i$, and fix $v_j^i\notin \{v_{d+1}^1,\ldots,v_{d+1}^\theta\}$. Then, $v_j^i$ is adjacent to $v_{d+1}^{i'}$ for $i'\neq i$, implying that $c\bigl(v_j^i\bigr) = i$. This implies that there is one way to color $G_{k,\theta}^d$, up to permutation of the coloring classes.

    \item This claim is true by construction.

    \item We prove this claim by constructing linearly independent eigenfunctions for every eigenvalue.
    \begin{itemize}
        \item For the eigenvalue $k/(k-1)$, we consider $d-1$ linearly independent functions $f_{1j}'$ for $2\leq j\leq d$, defined by
        \[
        f_{1j}'(v) = \begin{cases}
            1, &\text{ if } v = v_1^i, 1\leq i\leq \theta,\\
            -1, &\text{ if } v = v_j^i, 1\leq i\leq \theta,\\
            0, &\text{ otherwise.}
        \end{cases}
        \]
By Proposition \ref{prop:plusminus}, these are  eigenfunctions corresponding to the eigenvalue  $k/(k-1)$.

        \item For the eigenvalue $(k^2-d)/(k(k-1))$, one can check that one eigenfunction is given by
        \[
        f(v) \coloneqq \begin{cases}
            -k(k-d), &\text{ if } v = v_j^i, 1\leq i\leq \theta, j \leq d,\\
            d(k-1), &\text{ if } v = v_j^i, 1\leq i\leq \theta, j>d.
        \end{cases}
        \]
        \item For the eigenvalue $\theta/(\theta-1)$, we have $\theta-1$ linearly independent eigenfunctions $f_{1i}$, for $2\leq i\leq \theta$, where $f_{1i}$ is defined as in Definition \ref{def:f_ij} with respect to the coloring classes $V_i$.
        \item For the eigenvalue $1$, we have $(k-d-1)\theta$ linearly independent eigenfunctions $f_{v_2^i,v_j^i}$ for $1\leq i\leq \theta$ and $d+2\leq j\leq k$, as in Definition \ref{def:f_vw}:
        \[
        f_{v_{d+1}^i,v_j^i}(v) \coloneqq \begin{cases}
            1,&\text{ if }v = v_{d+1}^i,\\
            -1, &\text{ if }v = v_j^i,\\
            0, &\text{ otherwise.}
        \end{cases}
        \]
        One can check that these are eigenfunctions 
    by observing that $v_{d+1}^i$ and $v_j^i$ are duplicate vertices if $j>d+1$, and by applying Proposition \ref{prop:plusminus}.
        \item For the eigenvalue $1-(k-d)/(k(k-1)(\theta-1))$, we have $\theta-1$ linearly independent eigenfunctions $g_{1i}$ for $2\leq i\leq \theta$, defined by
        \[
        g_{1i}(v) \coloneqq \begin{cases}
            k(k-d), &\text{ if } v = v_j^1, j \leq d,\\
            -d(k-1), &\text{ if } v = v_j^1, j>d,\\
            -k(k-d), &\text{ if } v = v_j^i, j \leq d,\\
            d(k-1), &\text{ if } v = v_j^i, j>d,\\
            0, &\text{ otherwise.}
        \end{cases}
        \] 
        \item For the eigenvalue $1-1/(k-1)(\theta-1)$, we have $(d-1)(\theta-1)$ linearly independent eigenfunctions $h_{ij}$, for $2\leq i\leq \theta$ and $2\leq j\leq d$, defined by
        \[
        h_{ij}(v) := \begin{cases}
            1, &\text{ if  $v=v_1^1$ or $v=v_j^i$},\\
            -1, &\text{ if $v = v_j^1$ or $v=v_1^i$},\\
            0, &\text{ otherwise.}
        \end{cases}
        \]
        One can check that these are eigenfunctions by applying Proposition \ref{prop:plusminus}.
    \end{itemize}
    \item The eigenfunctions in this case are given by the same functions as in point (v).\qedhere
    \end{enumerate}
    \end{proof}
An immediate corollary is the following.
\begin{corollary}\label{cor:G_{k,chi}}
    Let $k,\theta>1$ and $0< d\leq k$ such that $k,\theta$ and $d$ do not all equal $2$. Assume that $k\geq \theta$ if $d=k$. Then $G_{k,\theta}^d$ has coloring number $\theta$, and we have the following cases for its largest eigenvalue.
    \begin{enumerate}
        \item If $\theta<k$, then $G_{k,\theta}^d$ has largest eigenvalue $\theta/(\theta-1)$ with multiplicity $\theta-1$.
        \item If $\theta=k$ and $d>1$, then $G_{k,\theta}^d$ has largest eigenvalue $\theta/(\theta-1)=k/(k-1)$ with multiplicity $\theta+d-2$.
        \item If $\theta=k$ and $d=1$, then $G_{k,\theta}^d$ has largest eigenvalue $\theta/(\theta-1)$ with multiplicity $\theta-1$.
        \item If $\theta=k+1$ and $d=1$, then $G_{k,\theta}^d$ has largest eigenvalue $\theta/(\theta-1)$ with multiplicity $\theta$.
        \item If $\theta>k+1$ and $d=1$, then $G_{k,\theta}^d$ has largest eigenvalue $(k+1)/k>\theta/(\theta-1)$ with multiplicity $1$.
        \item If $\theta>k>d>1$, then $G_{k,\theta}^d$ has largest eigenvalue $k/(k-1)>\theta/(\theta-1)$ with multiplicity $d-1$.
    \end{enumerate}
    
\end{corollary}

In particular, the first four cases in Corollary \ref{cor:G_{k,chi}} give us graphs with largest eigenvalue $\chi/(\chi-1)$.
    The last two cases give us graphs for which Question \ref{qu:equality} does not hold. Furthermore, the second case with $d<k$ and the fourth case give us graphs for which the converse of Proposition \ref{prop:multlambdaN} is not true.

\section{Constructing graphs with largest eigenvalue \texorpdfstring{$\chi/(\chi-1)$}{chi/(chi-1)}}\label{section:1-sum}

\subsection{Preliminary definitions and results}
Throughout this section, we fix two graphs $G_1$ and $G_2$, as well as vertices $x_1\in V(G_1)$ and $x_2\in V(G_2)$. We shall consider a graph operation, called the \emph{$1$-sum} \cite{Gurski2017} or \emph{graph joining} \cite{BanerjeeJost2008} or \emph{coalescing} \cite{CvetkovicDragos2010spectra}, which can be applied to two graphs that have the same largest eigenvalue, to obtain a new graph with this same largest eigenvalue. In particular, we shall see that, if $G_1$ and $G_2$ have the same coloring number $\chi$ and largest eigenvalue $$\lambda_{\max}(G_1) = \lambda_{\max}(G_2) = \frac{\chi}{\chi-1},$$ we can apply this operation to obtain a new graph with largest eigenvalue $\chi/(\chi-1)$.  Furthermore, we shall give the multiplicity of the eigenvalue $\chi/(\chi-1)$ of the $1$-sum of $G_1$ and $G_2$ in terms of its multiplicity for $G_1$ and $G_2$.

We start by giving the definition of the $1$-sum. The idea is that $G_1[x_1]\oplus G_2[x_2]$ is defined as the union of $G_1$ and $G_2$ in which the vertices $x_1$ and $x_2$ are identified in a new vertex $y$.
\begin{definition}\label{def:1-sum}
    The \emph{1-sum} $G_1[x_1]\oplus G_2[x_2]$ of $G_1$ and $G_2$ with respect to $x_1$ and $x_2$ is the graph defined by
    \begin{align*}
        V\biggl( G_1[x_1]\oplus G_2[x_2] \biggr) &:= V(G_1)\cup V(G_2) \cup \{y\} \setminus \{x_1,x_2\},\\
        E\biggl( G_1[x_1]\oplus G_2[x_2] \biggr) &:= E(G_1)\cup E(G_2) \cup\bigl\{\{y,v_i\}\colon \{v_i,x_i\}\in E(G_i), i =1,2\bigr\}\\
        &\quad \setminus \bigl\{\{x_i,v_i\} \colon v_i\in V(G_i), i=1,2\bigr\}.
    \end{align*}
        \end{definition}

    Note that the $1$-sum depends on the choice of $x_1$ and $x_2$, as is illustrated in Figure \ref{fig:1-sum}. If every choice of $x_1$ results in the same graph $G_1[x_1]\oplus G_2[x_2]$, then we also use the notation $G_1\oplus G_2[x_2]$.

    The definition of the $1$-sum of two graphs can be extended to the $1$-sum of $m$ graphs. 
    
    \begin{definition}\label{def:gen1-sum}
        For $1\leq i\leq m$, let $G_i$ be a graph and let $x_i\in V(G_i)$. The \emph{$1$-sum} of $G_1,\ldots,G_m$ with respect to $x_1,\ldots,x_m$, is the graph $\bigoplus_{i=1}^m G_i[x_i]$, defined by
    \begin{align*}
        V\biggl( \bigoplus_{i=1}^m G_i[x_i] \biggr) &:= \bigcup_{i=1}^m V(G_i) \cup \{y\}\setminus \{x_i\colon 1\leq i\leq m\},\\
        E\biggl( \bigoplus_{i=1}^m G_i[x_i] \biggr) &:= \bigcup_{i=1}^m V(G_i) \cup\bigl\{\{y,v_i\}\colon \{v_i,x_i\}\in E(G_i), 1\leq i\leq m\bigr\}\\
        &\quad \setminus \bigl\{\{x_i,v_i\} \colon v_i\in V(G_i), 1\leq i\leq m\bigr\}.
    \end{align*}
\end{definition}

\begin{figure}
    \centering
    \includegraphics[width=100mm]{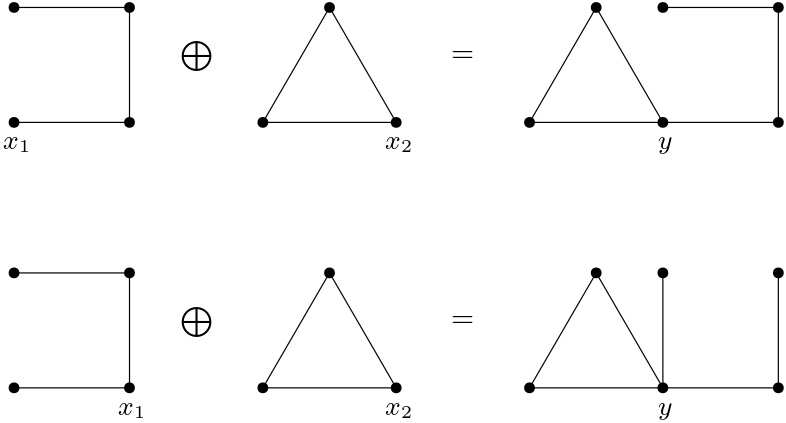}
    \caption{Two examples of the $1$-sum of two graphs with respect to $x_1$ and $x_2$}
    \label{fig:1-sum}
\end{figure}

\begin{remark}
   In Definition \ref{def:gen1-sum}, for the $1$-sum of $G_1,\ldots,G_m$ we identify one vertex of each graph $G_i$ with the same vertex in $\bigoplus_{i=1}^mG_i[x_i]$.
\end{remark}

\begin{remark}\label{rmk:colnmbr}
    It can be easily seen that
    \[
    \chi\bigl( G_1[x_1]\oplus G_2[x_2]\bigr) = \max\bigl\{ \chi(G_1),\chi(G_2)\bigr\}.
    \]
\end{remark}

Another graph operation that we shall consider is the \emph{join}.

\begin{definition}\label{def:join}
    The \emph{join} of $G_1$ and $G_2$, denoted $G_1\vee G_2$, is the graph constructed by taking the disjoint union of $G_1$ and $G_2$, and adding all edges between $V(G_1)$ and $V(G_2)$.
\end{definition}

\begin{example}\label{ex:genpetalgraph}
    For $n\geq2$, one can consider the $1$-sum $K_n^{(1)}\oplus K_n^{(2)}$ of two disjoint copies $K_n^{(1)}$ and $K_n^{(2)}$ of the complete graph on $n$ nodes. More generally, one can consider the $1$-sum of $m$ copies $K_n^{(i)}$ of $K_n$, denoted by
    \[
    \bigoplus_{i=1}^m K_n^{(i)},
    \]
    where we do not indicate with respect to what vertices $x_i$ we take the $1$-sum, as the choice of the vertices does not matter in this case.
    This gives one way of constructing the \emph{generalized petal graph} (see Figure \ref{fig:genpetalgphs} for two examples), which can be equivalently defined as
    \[
    K_1\vee mK_{n-1}.
    \]
As we shall see, defining the generalized petal graph as the $1$-sum of complete graphs, instead of the join of complete graphs, will allow us to infer that its largest eigenvalue must be $\chi/(\chi-1) = n/(n-1)$, and to compute its multiplicity, without having to calculate the whole spectrum. Furthermore, we shall generalize this result to the $1$-sum of arbitrary graphs.
\end{example}
\begin{figure}
    \centering
    \subfigure[The petal graph $\bigoplus_{i=1}^6K_3^{(i)}$]{\includegraphics[width=60mm]{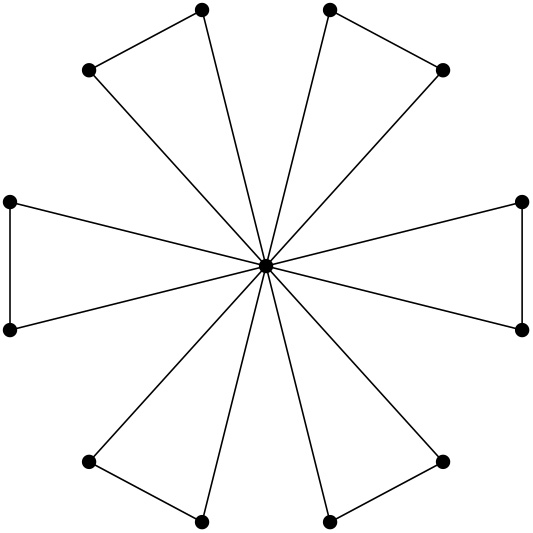}\label{fig:petalgph}}
    \hfill
    \subfigure[The flying kite graph $\bigoplus_{i=1}^4K_4^{(i)}$]{\includegraphics[width=60mm]{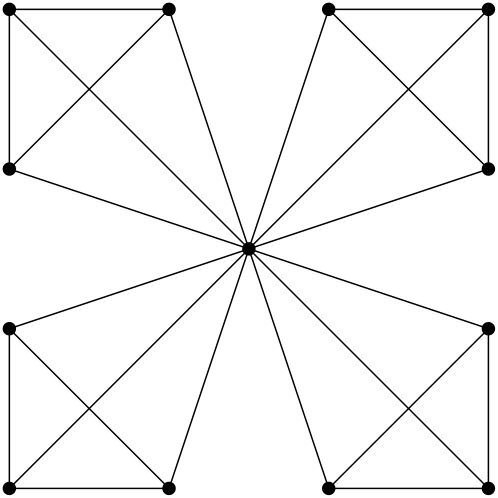}\label{fig:flyingkitegph}}
    \caption{Two examples of generalized petal graphs}
    \label{fig:genpetalgphs}
\end{figure}

For the rest of this section, in addition to fixing $G_1$, $G_2$, $x_1\in V(G_1)$ and $x_2\in V(G_2)$, we also let $y$ be the vertex of $G_1[x_1]\oplus G_2[x_2]$ that is identified with $x_1$ and $x_2$. 
We identify the subgraph of $G_1[x_1]\oplus G_2[x_2]$ induced by $V(G_1)\cup\{y\}\setminus \{x_1\}$ with $G_1$, and the subgraph induced by $V(G_2)\cup\{y\}\setminus\{x_2\}$ with $G_2$. \\

We shall also need the following definition.

\begin{definition}   Let $G'$ be a subgraph of $G$. Given a function $f\colon V(G) \to \R$, its \emph{restriction to $G'$} is defined as the function $f|_{G'}\colon V(G')\to \R$, given by $f|_{G'}(v) := f(v)$ for all $v\in V(G')$.
\end{definition}
 The following definition allows us to glue two functions together when taking the $1$-sum of two graphs.
\begin{definition}\label{def:gluefunctions}
    Let $f^1\colon V(G_1)\to\R$ and $f^2\colon V(G_2)\to\R$ be two functions such that $f^1(x_1) = f^2(x_2)$. Then we let 
    \[
    f^1\oplus_{x_1,x_2} f^2\colon V(G_1[x_1]\oplus G_2[x_2])\to \R
    \]
    be the function such that
    \[
    \bigl( f^1\oplus_{x_1,x_2} f^2\bigr)|_{G_1} = f^1 \text{ and } \bigl( f^1\oplus_{x_1,x_2} f^2\bigr)|_{G_2} = f^2.
    \]
\end{definition}

Furthermore, for $i=1,2$, we fix the notation

\[
\boldsymbol{0}^i\colon V(G_i)\to\R
\]
to denote the zero function.\\

We conclude with the following elementary lemma that will be needed in the proofs of this section.

\begin{lemma}\label{lem:fractions}
    Let $a,b,c,d\in \R_{>0}$. We have that
    \[
    \frac{a+b}{c+d} \leq \max\biggl \{ \frac ac,\frac bd\biggr\} \quad \text{ and } \quad \frac{a+b}{c+d} \geq \min\biggl \{ \frac ac,\frac bd\biggr\}.
    \]
    Moreover, equality holds if and only if $a/c = b/d$.
\end{lemma}

\subsection{Spectral properties of the \texorpdfstring{$1$}{1}-sum of graphs}\label{section:1-sum-spectrum}
Cvetković, Rowlinson, and Simić \cite{CvetkovicDragos2010spectra} gave the adjacency characteristic polynomial of the $1$-sum of two graphs (Theorem 2.2.3). Guo, Li and Shiu (2013) \cite{GuoLiShiu2013} gave the characteristic polynomial of the Kirchoff Laplacian (Corollary 2.3), signless Laplacian (Corollary 2.8) and normalized Laplacian (Corollary 3.3) of the $1$-sum of two graphs.
In \cite{BanerjeeJost2008}, Banerjee and Jost (2008) proved the following theorem.
\begin{theorem}[Theorem 2.5 from \cite{BanerjeeJost2008}]
    Assume that $\lambda$ is an eigenvalue of both $G_1$ and $G_2$, and that there exist corresponding eigenfunctions $f_\lambda^1$ and $f_\lambda^2$, such that $f_\lambda^1(p_1)=f_\lambda^2(p_2)=0$ for some $p_1\in V(G_1)$ and $p_2\in V(G_2)$. Then, the graph $G_1[p_1]\oplus G_2[p_2]$ also has eigenvalue $\lambda$, with an eigenfunction given by $f_\lambda^1\oplus_{p_1,p_2}f_\lambda^2$.
\end{theorem}
Whereas we are interested in the largest eigenvalue of $G_1[x_1]\oplus G_2[x_2]$,
Banerjee and Jost in \cite{BanerjeeJost2008} were mostly interested in constructing graphs which have eigenvalue $\lambda=1$.
They observed that, if $G_1$ and $G_2$ both have eigenvalue $\lambda=1$ with corresponding eigenfunctions $f_1^1$ and $f_1^2$, then one only has to require that $f_1^1(x_1) = f_1^2(x_2)$, for $f_1^1\oplus_{x_1,x_2}f_1^2$ to be an eigenfunction of $G_1[x_1]\oplus G_2[x_2]$.
We now generalize this to arbitrary eigenvalues.

\begin{proposition}\label{prop:gluingeigen}
Assume that $G_1$ and $G_2$ have a common eigenvalue $\lambda$, and that there exist corresponding eigenfunctions $f^i\colon V(G_i)\to \R$, for $i=1,2$, such that  $f^1(x_1) = f^2(x_2)$. Then,  $f^1\oplus_{x_1,x_2} f^2$ is an eigenfunction for $G_1[x_1]\oplus G_2[x_2]$ with eigenvalue $\lambda$.
\end{proposition}
\begin{proof}
    Let $G\coloneqq G_1[x_1]\oplus G_2[x_2]$, $f\coloneqq f^1\oplus_{x_1,x_2}f^2$ and $d_i\coloneqq \deg_{G_i} x_i$, for $i=1,2$.
    For $i=1,2$, and for every vertex $v_i\in V(G)\setminus\{y\}$ such that $v_i\in V(G_i)$, we have that
    \begin{align*}
        (1-\lambda)f(v_i) &= \frac1{\deg_{G_i} v_i} \left( \sum_{\substack{w_i\in G_i\\ w_i\sim v_i}}f^i(w_i) \right)
        = \frac1{\deg_{G} v_i} \left( \sum_{\substack{w\in G\\ w\sim v_i}}f(w) \right).
    \end{align*}
    Furthermore,
    \begin{align*}
        \frac1{\deg_G y}\Biggl( \sum_{\substack{w\in G\\ w\sim y}}f(w) \Biggr)
        &= \frac{d_1}{d_1+d_2} \cdot \frac 1{d_1}\left( \sum_{\substack{w_1\in G_1\\ w_1\sim x_1}} f(w_1)\right) + \frac{d_2}{d_1+d_2} \cdot \frac1{d_2}\left(\sum_{\substack{w_2\in G_2\\ w_2\sim x_2}} f(w_2) \right)\\
        &= \frac{d_1}{d_1+d_2}(1-\lambda)f^1(x_1) + \frac{d_2}{d_1+d_2}(1-\lambda)f^2(x_2)\\
        &= (1-\lambda)f(y).
    \end{align*}
   By \eqref{eq:eigenpair}, it follows that $f$ is an eigenfunction for $G$ with eigenvalue $\lambda$.
\end{proof}
We also have the following proposition about the gluing of functions.

\begin{proposition}\label{prop:gluingzero}
    Assume that, for some eigenvalue $\lambda$ of $G_1$, there exists a corresponding eigenfunction $f^1\colon V(G_1)\to\R$ such that $f^1(x_1)=0$. Then, $f^1\oplus_{x_1,x_2}\boldsymbol 0^2$ is an eigenfunction for $G_1[x_1]\oplus G_2[x_2]$ with eigenvalue $\lambda$.
\end{proposition}
\begin{proof}
    This is easily checked using Equation \eqref{eq:eigenpair}.
\end{proof}

We can use Proposition \ref{prop:gluingeigen} and \ref{prop:gluingzero} to give a lower bound for the multiplicity of $\lambda$ as an eigenvalue of $G_1[x_1]\oplus G_2[x_2]$, as follows.
\begin{theorem}\label{thm:lambdamult}
    For every $\lambda\in[0,2]$, we have that
    \[
    m_{G_1[x_1]\oplus G_2[x_2]}\bigl(\lambda\bigr) \geq m_{G_1}\bigl(\lambda\bigr) + m_{G_2}\bigl(\lambda\bigr)-1.
    \]
\end{theorem}

\begin{proof}
    For simplicity, let 
    \begin{align*}
            m_1 &\coloneqq m_{G_1}\bigl(\lambda\bigr),\\
            m_2 &\coloneqq m_{G_2}\bigl(\lambda\bigr),\\
            m_{12} &\coloneqq m_{G_1[x_1]\oplus G_2[x_2]}\bigl(\lambda\bigr).
    \end{align*}
   Furthermore, let 
    \[
    \bigl\{f^1,g_1^1,\ldots,g^1_{m_1-1}\bigr\} \quad\text{and}\quad \bigl\{f^2,g_1^2,\ldots,g^2_{m_2-1}\bigr\}
    \]
    be (possibly empty) bases for the eigenspace of $\lambda$ as an eigenvalue of $G_1$ and $G_2$, respectively, such that $g^1_i(x_1)=0=g^2_j(x_2)$ for $1\leq i\leq m_1-1$ and $1\leq j\leq m_2-1$. Then, by Proposition \ref{prop:gluingzero}, the functions
    \[
    g^1_i\oplus_{x_1,x_2}\boldsymbol 0^2 \quad\text{and}\quad  \boldsymbol 0^1 \oplus_{x_1,x_2} g^2_j
    \]
    are $m_1+m_2-2$ linearly independent eigenfunctions of $G_1[x_1]\oplus G_2[x_2]$ with eigenvalue $\lambda$. Hence, if we construct one more eigenfunction, we are done. We consider two cases.
    \begin{itemize}
        \item[Case 1:] $f^1(x_1)=0$ or $f^2(x_2)=0$. In this case, $f^1\oplus_{x_1,x_2}\boldsymbol 0^2$ or $\boldsymbol 0^1 \oplus_{x_1,x_2} f^2$, respectively, is an eigenfunction of $G_1[x_1]\oplus G_2[x_2]$ with eigenvalue $\lambda$, and it is linearly independent from the $m_1+m_2-2$ eigenfunctions that we exhibited above.
        \item[Case 2:] $f^1(x_1)\neq 0$ and $f^2(x_2)\neq 0$. In this case we can assume, without loss of generality, that $f^1(x_1)=f^2(x_2)$.  By Proposition \ref{prop:gluingeigen}, the function $f^1\oplus_{x_1,x_2}f^2$ is an eigenfunction for $G_1[x_1]\oplus G_2[x_2]$ with eigenvalue $\lambda$, and it is linearly independent from the $m_1+m_2-2$ eigenfunctions that we exhibited above.
    \end{itemize}

    This concludes the proof.
\end{proof}

\begin{remark}
The inequality in Theorem \ref{thm:lambdamult} is not always an equality. To see this, consider the $m$-petal graph from Section \ref{subsec:listofgphs}, which can be seen as the $1$-sum of copies of $K_3$ (cf.\ Example \ref{ex:genpetalgraph}). This graph has eigenvalue $1/2$, which is not an eigenvalue of $K_3$.\end{remark}

At the end of this section we shall compute the multiplicity of the eigenvalue
\[
\lambda = \max\bigl\{ \lambda_{\max}(G_1), \lambda_{\max}(G_2)\bigr\}
\]
for $G_1[x_1]\oplus G_2[x_2]$, by looking at eigenfunctions. This will be a consequence of the theorem below,
which states that the largest eigenvalue of  $G_1[x_2]\oplus G_2[x_2]$ is bounded above by the largest eigenvalues of both $G_1$ and $G_2$.
This interlacing result complements the ones in \cite{butler-interlacing} and, to the best of our knowledge, it has not been proved before. Notably, Atay and B\i y\i ko\u{g}lu (2005) \cite{KirchoffLaplacian1-sum2005} proved a similar result to Equation \eqref{eq:1-sum}, but it is the inverse inequality, and it involves the eigenvalues of the Kirchoff Laplacian.
\begin{theorem}\label{thm:1-sum}
   We have that 
    \begin{equation}\label{eq:1-sum}
    \lambda_{\max}\bigl(G_1[x_1]\oplus G_2[x_2]\bigr) \leq \max\biggl\{ \lambda_{\max}(G_1),\lambda_{\max}(G_2)\biggr\}.
    \end{equation}
\end{theorem}

\begin{proof}
    Let $G$ denote $G_1[x_1]\oplus G_2[x_2]$ for simplicity.
    Let $f$ be such that $\RQ(f) = \lambda_{\max}\bigl(G\bigr)$, and let $f^i\colon V_i\to \R$ be the restriction of $f$ to $G_i$, for $i=1,2$. If $f^1 = \boldsymbol{0}^1$ , then the statement follows immediately, because in this case $\RQ(f) = \RQ(f^2)$. If $f^2 = \boldsymbol{0}^2$ , then the statement follows analogously. Otherwise, we can use Lemma \ref{lem:fractions} to infer that
    \begin{align*}
        \lambda_{\max}\bigl(G\bigr) &= \RQ(f) \\
        &= \frac{\sum_{\substack{v,w\in V\bigl(G\bigr)\colon\\ v\sim w}}\biggl(f(v)-f(w)\biggr)^2}{\sum_{v\in V\bigl(G\bigr)}\deg_G v f(v)^2}\\
        &= \frac{\sum_{\substack{v,w\in V(G_1)\colon\\ v\sim w}}\biggl(f(v)-f(w)\biggr)^2 + \sum_{\substack{v,w\in V(G_2)\colon\\ v\sim w}}\biggl(f(v)-f(w)\biggr)^2}{\sum_{v\in V(G_1)}\deg_{G_1} v f(v)^2 + \sum_{v\in V(G_2)}\deg_{G_2} v f(v)^2}\\
        &\leq \max\left\{ \frac{\sum_{\substack{v,w\in V(G_1)\colon\\ v\sim w}}\biggl(f(v)-f(w)\biggr)^2}{\sum_{v\in V(G_1)}\deg_{G_1} v f(v)^2},\frac{\sum_{\substack{v,w\in V(G_2)\colon\\ v\sim w}}\biggl(f(v)-f(w)\biggr)^2}{\sum_{v\in V(G_2)}\deg_{G_2} v f(v)^2}\right\}\\
        &= \max\{\RQ\bigl(f^1\bigr),\RQ\bigl(f^2\bigr)\}\\
        &\leq \max\biggl\{ \lambda_{\max}(G_1),\lambda_{\max}(G_2)\biggr\}.\qedhere
    \end{align*}
\end{proof}

The following is an immediate corollary of Theorem \ref{thm:1-sum} and Theorem \ref{thm:3.1DS}.
\begin{corollary}\label{cor:gluing}
    Let $G_1$ and $G_2$ be two graphs with the same coloring number $\chi$, such that $\lambda_{\max}(G_1) = \lambda_{\max}(G_2) = \chi/(\chi-1)$.
    Then,
    \[
    \lambda_{\max}\bigl(G_1[x_1]\oplus G_2[x_2]\bigr) =\frac{\chi}{\chi-1}.
    \]
\end{corollary}

As a consequence of Corollary \ref{cor:gluing}, given any two graphs with the same coloring number $\chi$ and with largest eigenvalue $\chi/(\chi-1)$, we can construct a new graph which also has coloring number $\chi$ (by Remark \ref{rmk:colnmbr}) and largest eigenvalue $\chi/(\chi-1)$, by taking their $1$-sum with respect to any pair of vertices.\\

The following theorem tells us when the inequality in Theorem \ref{thm:1-sum} is an equality and, for this case, it also gives us the multiplicity of the largest eigenvalue.

\begin{theorem}\label{thm:1-summultiplicity}
    If $\lambda_{\max}(G_1)\geq \lambda_{\max}(G_2)$, then
    \begin{align*}
         &m_{G_1[x_1]\oplus G_2[x_2]}\bigl(\lambda_{\max}(G_1)\bigr) \in \\ &\quad \bigl\{ m_{G_1}\bigl(\lambda_{\max}(G_1)\bigr) + m_{G_2}\bigl(\lambda_{\max}(G_1)\bigr), m_{G_1}\bigl(\lambda_{\max}(G_1)\bigr) + m_{G_2}\bigl(\lambda_{\max}(G_1)\bigr)-1 \bigr\}.
    \end{align*}
            
    More specifically, we have the following two cases.
    \begin{enumerate}[label=(\arabic*)]
        \item Assume that, for $i=1,2$, for all $h_i\colon V(G_i)\to \R$ such that $\RQ_{G_i}(h_i) = \lambda_{\max}(G_1)$, we have that $h_1(x_1)=0$ and $h_2(x_2)=0$. In this case, we have that
        \[
        m_{G_1[x_1]\oplus G_2[x_2]}\bigl(\lambda_{\max}(G_1)\bigr) = m_{G_1}\bigl(\lambda_{\max}(G_1)\bigr) + m_{G_2}\bigl(\lambda_{\max}(G_1)\bigr).
        \]
        \item Otherwise, we have that
        \[
        m_{G_1[x_1]\oplus G_2[x_2]}\bigl(\lambda_{\max}(G_1)\bigr) = m_{G_1}\bigl(\lambda_{\max}(G_1)\bigr) + m_{G_2}\bigl(\lambda_{\max}(G_1)\bigr)-1.
        \]
    \end{enumerate}
\end{theorem}

\begin{proof}
   Let $G\coloneqq G_1[x_1]\oplus G_2[x_2]$, and let
    \begin{align*}
            m_1 &\coloneqq m_{G_1}\bigl(\lambda_{\max}(G_1)\bigr),\\
            m_2 &\coloneqq m_{G_2}\bigl(\lambda_{\max}(G_1)\bigr),\\
            m_{12} &\coloneqq m_{G_1[x_1]\oplus G_2[x_2]}\bigl(\lambda_{\max}(G_1)\bigr).
    \end{align*}
 Observe that, if $f\colon V(G)\to\R$ is an eigenfunction for $G$ with eigenvalue $\lambda_{\max}(G_1)$, then exactly one of the following is true:
    \begin{itemize}
        \item $f|_{G_1} = \boldsymbol 0^1$ and $\RQ(f|_{G_2}) = \lambda_{\max}(G_1)$,
        \item $\RQ(f|_{G_1}) = \lambda_{\max}(G_1)$ and $f|_{G_2} = \boldsymbol 0^2$, or
        \item $\RQ(f|_{G_1}) = \lambda_{\max}(G_1)$ and $\RQ(f|_{G_2}) = \lambda_{\max}(G_1)$.
    \end{itemize}

     From the min-max Principle it follows that, for at least one $i\in\{1,2\}$, $f|_{G_i}$ is an eigenfunction for $G_i$ with eigenvalue $\lambda_{\max}(G_1)$, and for at most one $i\in\{1,2\}$, $f|_{G_i}$ is the zero function on $G_i$. This allows us to give a basis for the eigenspace of $\lambda_{\max}(G_1)$ for $G$, in terms of bases for the eigenspace of this eigenvalue for $G_1$ and $G_2$.
    As in the proof of Theorem \ref{thm:lambdamult}, we let 
    \[
    \bigl\{f^1,g_1^1,\ldots,g^1_{m_1-1}\bigr\} \quad \text{and} \quad \bigl\{f^2,g_1^2,\ldots,g^2_{m_2-1}\bigr\}
    \]
    denote (possibly empty) bases for the eigenspace of $\lambda_{\max}(G_1)$
    as an eigenvalue of $G_1$ and $G_2$, respectively, such that $g_j^1(x_1)=0$ for $1\leq j\leq m_1-1$ and $g_j^2(x_2) = 0$ for $1\leq j\leq m_2-1$.
    We consider two cases.
    \begin{enumerate}[label=(\arabic*)]
        \item If $f^1(x_1)=0$ and $f^2(x_2) = 0$, then 
        \begin{align*}
        &\bigl\{f^1\oplus_{x_1,x_2}\boldsymbol{0}^2,g^1_1\oplus_{x_1,x_2}\boldsymbol{0}^2,\ldots,g^1_{m_1-1}\oplus_{x_1,x_2}\boldsymbol{0}^2\bigr\} \cup \\ &\bigl\{\boldsymbol{0}^1\oplus_{x_1,x_2}f^2,\boldsymbol{0}^1\oplus_{x_1,x_2}g_1^2,\ldots,\boldsymbol{0}^1\oplus_{x_1,x_2}g^2_{m_2-1}\bigr\}
        \end{align*}
        is a basis for the eigenspace of  the eigenvalue $\lambda_{\max}(G_1)$ for $G$ of size $m_1+m_2$.
        \item If one of $f^1(x_1)$ and $f^2(x_2)$ is non-zero, then
        we have one of the following three subcases.
        \begin{enumerate}[label=(\roman*)]
            \item If $f^1(x_1)=0$ and $f^2(x_2)\neq 0$, then
            \[
            \bigl\{f^1\oplus_{x_1,x_2}\boldsymbol{0}^2,g^1_1\oplus_{x_1,x_2}\boldsymbol{0}^2,\ldots,g^1_{m_1-1}\oplus_{x_1,x_2}\boldsymbol{0}^2\bigr\} \cup \bigl\{\boldsymbol{0}^1\oplus_{x_1,x_2}g_1^2,\ldots,\boldsymbol{0}^1\oplus_{x_1,x_2}g^2_{m_2-1}\bigr\}
            \]
            is a basis for the eigenspace of $\lambda_{\max}(G_1)$ for $G$ of size $m_1+m_2-1$.
            
            \item Analogously, if $f^1(x_1)\neq 0$ and $f^2(x_2)= 0$, then
            \[
            \bigl\{g^1_1\oplus_{x_1,x_2}\boldsymbol{0}^2,\ldots,g^1_{m_1-1}\oplus_{x_1,x_2}\boldsymbol{0}^2\bigr\} \cup \bigl\{\boldsymbol{0}^1\oplus_{x_1,x_2}f^2,\boldsymbol{0}^1\oplus_{x_1,x_2}g_1^2,\ldots,\boldsymbol{0}^1\oplus_{x_1,x_2}g^2_{m_2-1}\bigr\}
            \]
            is a basis for the eigenspace of $\lambda_{\max}(G_1)$ for $G$ of size $m_1+m_2-1$.
            
            \item If $f^1(x_1)\neq 0$ and $f^2(x_2)\neq 0$, then we  assume that $f^1(x_1) = f^2(x_2)$, and we have that
            \[
            \bigl\{f^1\oplus_{x_1,x_2}f^2\bigr\}\cup\bigl\{g^1_1\oplus_{x_1,x_2}\boldsymbol{0}^2,\ldots,g^1_{m_1-1}\oplus_{x_1,x_2}\boldsymbol{0}^2\bigr\} \cup \bigl\{\boldsymbol{0}^1\oplus_{x_1,x_2}g_1^2,\ldots,\boldsymbol{0}^1\oplus_{x_1,x_2}g^2_{m_2-1}\bigr\}
            \]
            is a basis for the eigenspace of $\lambda_{\max}(G_1)$ for $G$ of size $m_1+m_2-1$.
        \end{enumerate}
        In all three subcases, we have that $m_{12}=m_1+m_2-1$.\qedhere
    \end{enumerate}
\end{proof}

As a corollary of Theorem \ref{thm:1-summultiplicity}, we can now give the multiplicity of the eigenvalue $\chi/(\chi-1)$ for the $1$-sum of two graphs that have coloring number $\chi$ and largest eigenvalue $\chi/(\chi-1)$.
\begin{corollary}\label{cor:1-summult}
    Let $G_1$ and $G_2$ be graphs that have the same coloring number $\chi$ and largest eigenvalue $$\lambda_{\max}(G_1) = \lambda_{\max}(G_2) = \frac{\chi}{\chi-1}.$$
    Then, the multiplicity of the largest eigenvalue $\chi/(\chi-1)$ of $G_1[x_1]\oplus G_2[x_2]$ is
    \[
    m_{G_1[x_1]\oplus G_2[x_2]}\biggl(\frac{\chi}{\chi-1}\biggr) = m_{G_1}\biggl(\frac{\chi}{\chi-1}\biggr) + m_{G_2}\biggl(\frac{\chi}{\chi-1}\biggr) - 1.
    \]
    Furthermore, the eigenfunctions corresponding to $\lambda_{\max}\bigl(G_1[x_1]\oplus G_2[x_2]\bigr)$ are precisely the non-zero functions $f$ such that $f|_{G_1}$ is either the zero function or an eigenfunction for $G_1$ with eigenvalue $\chi/(\chi-1)$, and $f|_{G_2}$ is either the zero function or an eigenfunction for $G_2$ with eigenvalue $\chi/(\chi-1)$.
\end{corollary}
\begin{proof}
    We are in the setting of Case (2)(iii) in the proof of Theorem \ref{thm:1-summultiplicity}, 
    since we have that $f_{ij}^1\colon V(G_1)\to\R$ from Definition \ref{def:f_ij} is non-zero on $x_1$, for $x_1\in V_i$, and we have that $f_{kl}^2\colon V(G_2)\to\R$ is non-zero on $x_2$, for $x_2\in V_k$.
    This proves the first part of the corollary. The second part follows by looking at the basis that is given in Case (2)(iii) in the proof of Theorem \ref{thm:1-summultiplicity}.
\end{proof}
\begin{example}\label{ex:multN-2}
   Consider the generalized petal graphs from Example \ref{ex:genpetalgraph}. As a consequence of Corollary \ref{cor:1-summult} we have that, for $n\geq 2$, the graph $$G\coloneqq\bigoplus_{i=1}^m K_n^{(i)},$$
    which has coloring number $\chi = n$, has largest eigenvalue
    \[
    \lambda_{\max}(G) = \frac{\chi}{\chi-1}
    \]
with multiplicity $m(n-1)-m+1 = |V(G)|-m$.

As two particular cases (for $m=1$ and $m=2$, respectively),
\begin{itemize}
    \item The complete graph $K_N$ has largest eigenvalue $\chi/(\chi-1)$ with multiplicity $N-1$, and it is well-known that this is the only connected graph with an eigenvalue that has multiplicity $N-1$. Therefore, the complete graph is the only graph that has largest eigenvalue $\chi/(\chi-1)$ with multiplicity $N-1$.
    \item By Proposition 8 in \cite{VanDamOmidi2011}, the generalized petal graph $K_n^{(1)}\oplus K_n^{(2)}$ is the only graph that has largest eigenvalue $\chi/(\chi-1)$ with multiplicity $N-2$.
\end{itemize}    
\end{example}
In Example \ref{ex:multN-2}, we characterized graphs with largest eigenvalue $\chi/(\chi-1)$ whose multiplicity equals $N-1$ or $N-2$, respectively. Graphs with $\lambda_N=\chi/(\chi-1)$ whose multiplicity equals $N-3$ have also been characterized, and this result can be found in \cite{TianWang2021}.
We may ask the following question:
\begin{question}\label{qu:multN-4}
    Which graphs have largest eigenvalue $\chi/(\chi-1)$ with multiplicity $N-k$, where $k\geq4$ is relatively small compared to $N$?
\end{question}
\begin{remark}
    The $1$-sum of two complete graphs of the same size, $K_n^{(1)}\oplus K_n^{(2)}$, also has the property that its largest eigenvalue equals $(N+1)/(N-1)$. The only other graphs which have this property are complete graphs with one edge removed. All other non-complete graphs have largest eigenvalue strictly bigger than $(N+1)/(N-1)$, as proven by Sun and Das (2016) \cite{SunDas2016} and by Jost, Mulas and Münch (2021) \cite{JMM}.
\end{remark}

\subsection{Generalizing the \texorpdfstring{$1$}{1}-sum}

Different generalizations of the $1$-sum have been introduced in varying contexts. One of these is called the \emph{clique-sum}, the \emph{$k$-clique-sum} or the \emph{$k$-sum}, depending on the reference, and it is used, for example, in the proof of the Structure Theorem from Robertson and Seymour (2003) \cite{RobertsonSeymour} on the structure of graphs for which no minor is isomorphic to a fixed graph $H$. The idea of the $k$-clique-sum, for a positive integer $k$, is to first glue $G_1$ and $G_2$ together at a $k$-clique, and then remove either no, all or some of the edges of this $k$-clique in the new graph. 

However, for the $k$-clique-sum,we cannot generalize Theorem \ref{thm:1-sum}. To see this, consider a $2$-clique-sum of two copies of $K_3$, both having largest eigenvalue $\lambda_{\max}(K_3) = 3/2$. In this case, the $2$-clique-sum can  either be $C_4$ or $K_4\setminus \{e\}$ (depending on whether we remove the edge, i.e.\ the $2$-clique, in which we glue the graphs together). The former has largest eigenvalue $\lambda_{\max}(C_4) = 2$, and the latter has largest eigenvalue $\lambda_{\max}(K_4\setminus\{e\}) = 5/3$. Both these values are bigger than $\lambda_{\max}(K_3)$. In Theorem \ref{thm:1-sum} we saw, in contrast, that the opposite inequality is true for the $1$-sum.

The aim of this subsection is to offer a different generalization of the $1$-sum, for which we can prove a generalization of Theorem \ref{thm:1-sum}.

\begin{definition}\label{def:k-union}
    Let $G_1$ and $G_2$ be graphs such that \[E(G_1)\cap E(G_2) = \varnothing.\] Their \emph{edge-disjoint union} is the graph $G_1\sqcup_E G_2$ with vertex set \[V(G_1\sqcup_E G_2) := V(G_1)\cup V(G_2)\]
    and edge set
    \[E(G_1\sqcup_E G_2) := E(G_1)\sqcup E(G_2).\]
\end{definition}
From here on, we fix two graphs $G_1$ and $G_2$ such that $E(G_1)\cap E(G_2)=\varnothing$.
\begin{remark}
    If $|V(G_1)\cap V(G_2)|=0$, then the edge-disjoint union of $G_1$ and $G_2$ is simply the disjoint union of $G_1$ and $G_2$.
    If $|V(G_1)\cap V(G_2)|=1$, then the edge-disjoint union and the $1$-sum of $G_1$ and $G_2$ coincide.
\end{remark}
We extend our results on the gluing of eigenfunctions from the context of the $1$-sum to edge-disjoint unions. To achieve this, we first introduce notation that facilitates the gluing of functions in edge-disjoint unions. This new definition generalizes Definition \ref{def:gluefunctions}.
\begin{definition}\label{def:gluefunctionsgen}
    Let $f^1\colon V(G_1)\to\R$ and $f^2\colon V(G_2)\to\R$ be two functions such that, for all $v\in V(G_1)\cap V(G_2)$, we have that
    $f^1(v) = f^2(v)$. Then, we let 
    \[
    f^1 \sqcup_E f^2\colon V(G_1 \sqcup_E G_2)\to \R
    \]
    be the function such that
    \[
    \bigl( f^1 \sqcup_E f^2\bigr)|_{G_1} = f^1 \text{ and } \bigl( f^1 \sqcup_E f^2\bigr)|_{G_2} = f^2.
    \]
\end{definition}
We are now prepared to state the two propositions that generalize Proposition \ref{prop:gluingeigen} and Proposition \ref{prop:gluingzero}, respectively. The proofs of these propositions follow similarly to the proofs of their counterparts for the $1$-sum.
\begin{proposition}
    Assume that $G_1$ and $G_2$ have a common eigenvalue $\lambda$, and that there exist corresponding eigenfunctions $f^i\colon V(G_i)\to \R$, for $i=1,2$, such that  $f^1(v) = f^2(v)$ for all $v\in V(G_1)\cap V(G_2)$. Then,  $f^1\sqcup_E f^2$ is an eigenfunction for $G_1 \sqcup_E G_2$ with eigenvalue $\lambda$.
\end{proposition}
\begin{proposition}
    Assume that, for some eigenvalue $\lambda$ of $G_1$, there exists a corresponding eigenfunction $f^1\colon V(G_1)\to\R$ such that, for all $v\in V(G_1)\cap V(G_2)$, we have that $f^1(v) = 0$. Then, $f^1\sqcup_E \boldsymbol 0^2$ is an eigenfunction for $G_1\sqcup_E G_2$ with eigenvalue $\lambda$.
\end{proposition}

In the context of the edge-disjoint union, we can also establish the following generalization of Theorem \ref{thm:1-sum}.

\begin{theorem}\label{thm:edgedu}
    We have that
    \[
    \lambda_{\max}(G_1\sqcup_EG_2) \leq \max \biggl\{ \lambda_{\max}(G_1),\lambda_{\max}(G_2)\biggr\}.
    \]
\end{theorem}
\begin{proof}
    The proof is analogous to the proof of Theorem \ref{thm:1-sum}. The key ingredients are Lemma \ref{lem:fractions} and the fact that $G_1$ and $G_2$ are edge-disjoint.
\end{proof}
A special case of Theorem \ref{thm:edgedu} can be found in Lemma 15 from Coutinho, Grandsire and Passos (2019) \cite{Gabriel-colouring}.

Note that we cannot generalize the results from Section \ref{section:1-sum-spectrum} about the $1$-sum and the coloring number to the more general case of the edge-disjoint union. This is partly due to the fact that the following inequality is not always an equality:
\[
\chi(G_1\sqcup_EG_2) \geq \max\{\chi(G_1),\chi(G_2)\}.
\]
However, we make the extra assumption that $\chi(G_1\sqcup_EG_2) = \max\{\chi(G_1),\chi(G_2)\}$,  we can then generalize Corollary \ref{cor:gluing} to obtain the following corollary of Theorem \ref{thm:edgedu} and Theorem \ref{thm:edgespread}.
\begin{corollary}\label{cor:gluinggen}
    If $G_1$ and $G_2$ are two graphs with the same coloring number $\chi$ such that $\lambda_{\max}(G_1) = \lambda_{\max}(G_2) = \chi/(\chi-1)$ and $\chi(G_1\sqcup_EG_2) = \chi$,
    then
    \[
    \lambda_{\max}\bigl(G_1\sqcup_EG_2\bigr) =\frac{\chi}{\chi-1}.
    \]
\end{corollary}

\section{Upper bounds}\label{section:upperbounds}
In this section, we shall give some upper bounds on the largest eigenvalue $\lambda_N$ of a fixed graph $G$ which depend on its coloring number $\chi$. We start with a theorem for graphs that admit a proper $\chi$-coloring which has coloring classes of the same size, and we then generalize it to a theorem which applies to all graphs.
\begin{theorem}\label{thm:upper1}
     Let $\delta$ denote the smallest vertex degree of $G$.
     If there exists a proper $\chi$-coloring of the vertices for which all coloring classes have the same size, then
     \begin{equation*}
         \lambda_N\leq  \frac{N}{\delta}.
     \end{equation*}Moreover, the inequality is sharp.
\end{theorem}
\begin{proof}
 If $f$ is an eigenfunction for $\lambda_N$ and the coloring classes are denoted by $V_1,\ldots,V_\chi$, then
\begin{align*}
         \lambda_N(G)&= \RQ_G(f)\\
        &= \frac{\sum_{v\sim w}\biggl(f(v)-f(w)\biggr)^2}{\sum_{v\in V}\deg v\cdot f(v)^2}\\
    &=\frac{\sum_{i\neq j}\sum_{\substack{w_i\sim w_j,\\ w_i\in V_i, w_j\in V_j}}\biggl(f(w_i)-f(w_j)\biggr)^2}{\sum_{v\in V}\deg v\cdot f(v)^2}\\
    &\leq \frac{\sum_{i\neq j}\sum_{\substack{w_i\sim w_j,\\ w_i\in V_i, w_j\in V_j}}\biggl(f(w_i)-f(w_j)\biggr)^2}{\sum_{v\in V}\delta \cdot f(v)^2}\\
     &\leq \frac{\sum_{i\neq j}\sum_{w_i\in V_i, w_j\in V_j}\biggl(f(w_i)-f(w_j)\biggr)^2}{\sum_{v\in V}\delta \cdot f(v)^2}.
\end{align*}

Now, by assumption, each coloring class $V_i$ has size $N/\chi$. Let $\widehat{G}$ be the  complete multipartite graph with partition classes $V_1,\ldots,V_\chi$, and let $k:=N-N/\chi$. Then, $\widehat{G}$ is a $k$-regular graph and, by Theorem 3.6 in \cite{SunDas2020}, $\lambda_N(\widehat{G})=\chi/(\chi-1)$. Therefore,

\begin{align*}
         \lambda_N(G)
     &\leq \frac{\sum_{i\neq j}\sum_{w_i\in V_i, w_j\in V_j}\biggl(f(w_i)-f(w_j)\biggr)^2}{\sum_{v\in V}\delta \cdot f(v)^2}\\
     &= \frac{k}{\delta} \cdot \frac{\sum_{i\neq j}\sum_{w_i\in V_i, w_j\in V_j}\biggl(f(w_i)-f(w_j)\biggr)^2}{\sum_{v\in V}k \cdot f(v)^2} \\
     & = \frac{k}{\delta} \cdot\RQ_{\widehat{G}}(f)\\
     &\leq \frac{k}{\delta} \cdot \lambda_N(\widehat{G})\\
     &= \frac{k}{\delta} \cdot \frac{\chi}{\chi-1}\\
     &=\frac{N}{\delta}.
\end{align*}
This proves the inequality. It is sharp since it becomes an equality for complete multipartite graphs with coloring classes of the same size.
\end{proof}
 We now offer a generalization of Theorem \ref{thm:upper1} to all graphs.
\begin{theorem}
Fix a proper $\chi$-coloring with coloring classes $V_1,V_2,\ldots, V_\chi$ such that their cardinalities $N_i\coloneqq |V_i|$ satisfy $N_i\geq N_{i+1}$ for $1\leq i\leq \chi$. Let also
     \[
     x\coloneqq \min_{1\leq i\leq \chi, v\in V_i}\frac{\deg v}{N-N_i}.
     \]
     Then,
     \begin{equation*}
         \lambda_N\leq \frac1x \cdot \frac N{N-N_1}.
     \end{equation*}Furthermore, this inequality is sharp.
\end{theorem}

\begin{proof}
 If $f$ is an eigenfunction for $\lambda_N$, then
\begin{align*}
    \lambda_N(G)&= \RQ_G(f)\\
    &= \frac{\sum_{v\sim w}\biggl(f(v)-f(w)\biggr)^2}{\sum_{v\in V}\deg v\cdot f(v)^2}\\
    &=\frac{\sum_{i\neq j}\sum_{\substack{w_i\sim w_j,\\ w_i\in V_i, w_j\in V_j}}\biggl(f(w_i)-f(w_j)\biggr)^2}{\sum_{v\in V}\deg v\cdot f(v)^2}\\
    &\leq \frac{\sum_{i\neq j}\sum_{\substack{w_i\sim w_j,\\ w_i\in V_i, w_j\in V_j}}\biggl(f(w_i)-f(w_j)\biggr)^2}{x\sum_{\substack{1\leq i\leq \chi, \\ v\in V_i}}\bigl(N-N_i\bigr)f(v)^2}\\
     &\leq \frac{\sum_{i\neq j}\sum_{w_i\in V_i, w_j\in V_j}\biggl(f(w_i)-f(w_j)\biggr)^2}{x\sum_{\substack{1\leq i\leq \chi, \\ v\in V_i}}\bigl(N-N_i\bigr)f(v)^2}.
\end{align*}

By assumption, each coloring class $V_i$ has size $N_i$. Let $\widehat G\coloneqq K_{N_1,\ldots,N_\chi}$ be the  complete multipartite graph with partition classes $V_1,\ldots,V_\chi$. Then, by Theorem 3.5 in \cite{SunDas2020.2}, $$\lambda_N(\widehat G) \leq \frac{N}{N-N_1}.$$
Therefore,
\begin{align*}
    \lambda_N(G)
     &\leq \frac{\sum_{i\neq j}\sum_{w_i\in V_i, w_j\in V_j}\biggl(f(w_i)-f(w_j)\biggr)^2}{x\sum_{\substack{1\leq i\leq \chi, \\ v\in V_i}}\bigl(N-N_i\bigr)f(v)^2}\\
     &= \frac1x\cdot \frac{\sum_{i\neq j}\sum_{w_i\in V_i, w_j\in V_j}\biggl(f(w_i)-f(w_j)\biggr)^2}{\sum_{\substack{1\leq i\leq \chi, \\ v\in V_i}}\bigl(N-N_i\bigr)f(v)^2}\\
     &= \frac{1}{x} \cdot\RQ_{\widehat{G}}(f)\\
     &\leq \frac{1}x \cdot \lambda_N(\widehat{G})\\
     &\leq \frac1x \cdot \frac N{N-N_1}.
\end{align*} The inequality is sharp because of Theorem \ref{thm:upper1}.
\end{proof}

We shall now prove a theorem for graphs that satisfy the setting of Question \ref{qu:equality}. 
\begin{theorem}\label{thm:d-reg}
    Fix a proper $\chi$-coloring that has coloring classes $V_1,\ldots,V_\chi$. Assume that $G$ is $d$-regular, and that
    \[
    e\bigl(v,V_i\bigr) = \begin{cases}
    \frac{\deg v}{\chi-1}&\text{ if } v\notin V_i,\\
    0 &\text{ if }v\in V_i.
    \end{cases}
    \]
    Then,
    \[
    \lambda_N \leq \max\biggl\{ \frac{N}{d}\cdot\frac{\chi-1}{\chi},\frac{\chi}{\chi-1}\biggr\}.
    \]
\end{theorem}

\begin{proof}
    Let $f\notin \langle f_{1i}\colon 2\leq i\leq\chi\rangle$ be a non-constant function. Then, by Corollary \ref{cor:orthog}, for all $j$ with $1\leq j\leq \chi$ we have that
    \begin{equation}\label{eq:sumzero}
    \sum_{v\in V_i}f(v) = \frac1d \sum_{v\in V_i}\deg v f(v) = 0.
    \end{equation}
    This implies that
    \begin{align*}
        \RQ(f) &=\frac{\sum_{1\leq i<j\leq \chi}\sum_{\substack{v_i\sim v_j\colon\\v_i\in V_i\\v_j\in V_j}}\Bigl(f(v_i)-f(v_j)\Bigr)^2}{d\sum_{v\in V}f(v)^2}\\
        &\leq \frac{\sum_{1\leq i<j\leq \chi}\sum_{\substack{v_i\in V_i\\v_j\in V_j}}\Bigl(f(v_i)-f(v_j)\Bigr)^2}{d\sum_{v\in V}f(v)^2}\\
        &= \frac{\sum_{v\in V}\Bigl(N-N/\chi\Bigr)f(v)^2 - 2\sum_{1\leq i<j\leq \chi}\sum_{v_i\in V_i}f(v_i)\sum_{v_j\in V_j}f(v_j)}{d\sum_{v\in V}f(v)^2}\\
        &\overset{\eqref{eq:sumzero}}{=} \frac{\Bigl(N-N/\chi\Bigr)\sum_{v\in V}f(v)^2}{d\sum_{v\in V}f(v)^2}\\
        &= \frac{N-N/\chi}d\\
        &= \frac Nd\cdot \frac{\chi-1}{\chi}.
    \end{align*}    
    Furthermore, by Proposition \ref{prop:eigenfns}, we have that the functions $f_{1i}$'s from Definition \ref{def:f_ij} are eigenfunctions with corresponding eigenvalue $\chi/(\chi-1)$. Therefore, if $f$ is an eigenfunction such that $\RQ(f) = \lambda_N$, then either $\RQ(f) = \chi/(\chi-1)$, or $$\RQ(f)>\frac{\chi}{(\chi-1)}\quad \text{and}\quad\RQ(f) \leq \frac{N(\chi-1)}{d\chi}.$$
    This implies that
    \[
    \lambda_N \leq \max\biggl\{ \frac{N}{d}\cdot\frac{\chi-1}{\chi},\frac{\chi}{\chi-1}\biggr\}.\qedhere
    \]
\end{proof}

\begin{remark}
    Note that we cannot give a better upper bound than $2$ for $\lambda_N$ if our only information about a graph is its coloring number $\chi$. Consider, for example, the complete multipartite graph $G$ with coloring classes $V_1,\ldots,V_\chi$ of sizes
    \[
    |V_i| = \begin{cases}
        t, &\text{ if } i=1,\\
        1, &\text{ otherwise.}
    \end{cases}
    \]
    This graph is also known as a \emph{complete split graph}.    
    Note that $N = |V(G)| = t + \chi-1$.
    We know by Lemma 2.14 in \cite{SunDas2019} that the largest eigenvalue of $G$ equals
    \[
    \lambda_{N}(G) = 1+\frac{t}{N-1} = 2 - \frac{\chi-2}{N-1}
    \]
    and we have that
    \[
    \lim_{N\to\infty} \lambda_N(G) = 2 - \lim_{N\to\infty}\biggl(\frac{\chi-2}{N-1}\biggr) = 2.
    \]
    Since we can choose $t$ independently of $\chi$, we see that the best upper bound that we can give for $\lambda_N$ equals $2$.
\end{remark}

\bibliographystyle{plain} 

\bibliography{Bibliography}

\end{document}